\documentclass{amsart}
\usepackage[utf8]{inputenc}
\usepackage[dvipsnames]{xcolor}
\usepackage{amsfonts}
\usepackage{amssymb}
\usepackage{graphicx}
\usepackage{tikz}
\usepackage{tikz-cd}
\usetikzlibrary{positioning}
\usetikzlibrary{calc, positioning, fit, shapes.misc}
\usetikzlibrary{shapes.geometric}
\usepackage{fancyhdr}
\usepackage{bbm}
\usepackage{subcaption}
\usepackage[symbol]{footmisc}
\usepackage[colorlinks=true, allcolors=blue,linkcolor=blue, citecolor=blue]{hyperref}

\newtheorem{theorem}{Theorem}
\newtheorem{definition}[theorem]{Definition}
\newtheorem{example}[theorem]{Example}

\newtheorem{proposition}[theorem]{Proposition}
\newtheorem{lemma}[theorem]{Lemma}

\newcommand{\R}{\mathbb{R}}

\newcommand{\K}{\mathbb{K}}
\newcommand{\Z}{\mathbb{Z}}
\newcommand{\Set}{\mathrm{\textbf{Set}}}
\newcommand{\FinSet}{\mathrm{\textbf{FinSet}}}
\newcommand{\Simp}{\mathrm{\textbf{Simp}}}
\newcommand{\Ab}{\mathrm{\textbf{Ab}}}
\newcommand{\Po}{\mathrm{\textbf{Po}}}
\newcommand{\Hyp}{\mathrm{\textbf{Hyp}}}
\newcommand{\MHyp}{\mathrm{\textbf{MHyp}}}
\newcommand{\SSys}{\mathrm{\textbf{SSys}}}
\newcommand{\Rel}{\mathrm{\textbf{Rel}}}

\newcommand{\Bgroup}        {\mm{\sf B}}
\newcommand{\Cgroup}        {\mm{\sf C}}
\newcommand{\Dgroup}	    {\mathsf{D}}
\newcommand{\Hgroup}	     {\mathsf{H}}
\newcommand{\Zgroup}        {\mm{\sf Z}}

\newcommand{\PH}	     {\mm{\mathsf{PH}}}

\renewcommand{\hom}[2]{\Hgroup_{#2}^{#1}}
\newcommand{\betti}[2]{\beta_{#2}^{#1}}
\renewcommand{\H}{\mathcal{H}}
\newcommand{\bary}{\mathcal{B}}
\newcommand{\missing}{\mathcal{M}}
\newcommand{\po}{\mathcal{P}}
\newcommand{\ECP}{\mathcal{O}}
\newcommand{\pol}{\Gamma}
\newcommand{\Ord}{Ord}

\DeclareMathOperator{\im}{im}

\newcommand{\myker}{\mm{\mathrm{ker\,}}}
\newcommand{\myim}{\mm{\mathrm{im\,}}}

\newcommand{\bdr}{\mm{\mathrm{\partial}}}

\newcommand{\Nrv}{\mm{\mathrm{Nrv}}}

\newcommand {\mm}[1]        {\ifmmode{#1}\else{\mbox{\(#1\)}}\fi}

\newcommand{\Kspace}        {\mm{{\mathbb K}}}

\newcommand{\Acal}        {\mm{\mathcal A}}
\newcommand{\Rcal}        {\mm{\mathcal R}}
\newcommand{\Ical}        {\mm{\mathcal I}}

\newcommand{\ie}            {{i.e.}}
\newcommand{\etal}            {{et al.}}

\newcommand{\infgroup}{\mm{\mathsf{inf}}}

\title[A Survey of Homology Theories for Hypergraphs]{A Survey of Simplicial, Relative, and Chain Complex Homology Theories for Hypergraphs}

\author[E. Gasparovic, et al.]{Ellen Gasparovic}
\address{Union College, 807 Union Street, Schenectady, NY 12309}
\email{gasparoe@union.edu}
\author[]{Emilie Purvine}
\address{Pacific Northwest National Laboratory, 1100 Dexter Ave N., Seattle, WA 98109}
\email{emilie.purvine@pnnl.gov}
\author[]{Radmila Sazdanovic}
\address{NC State University, PO Box 8205, Raleigh NC 27695}
\email{rsazdan@ncsu.edu}
\author[]{Bei Wang}
\address{University of Utah, 72 South Central Campus Drive,
Salt Lake City, UT 84112}
\email{beiwang@sci.utah.edu}
\author[]{Yusu Wang}
\address{Halicioglu Data Science Institute at University of California, HDSI 446, San Diego, CA  92093}
\email{yusuwang@ucsd.edu}
\author[]{Lori Ziegelmeier}
\address{Macalester College, 1600 Grand Avenue, Saint Paul, MN 55104}
\email{lziegel1@macalester.edu}

\begin{document}

\begin{abstract}
Hypergraphs have seen widespread application in network and data science communities in recent years. 
We present a survey of recent work to construct auxiliary structures from hypergraphs---specifically simplicial, relative, and chain complexes---that can be used to build homology theories for hypergraphs. We define and describe nine different constructions and their associated homology theories.
We discuss some interesting properties of each homology theory to
show how various hypergraph properties imply properties of the homology groups. We also include discussion of functoriality for several of the homology theories. Finally, we provide a series of illustrative examples by computing many of these homology theories for small hypergraphs to show the variability of the methods and build intuition. 
\end{abstract}

\maketitle

\section{Introduction}
\label{sec:introduction}

Homology---uncovering the ``shape" of an object as represented by its multidimensional holes, which are preserved under continuous deformations like stretching and twisting---has been studied by theoretical mathematicians since the 1800s with the introduction of the Euler characteristic. 
The study of homology and homological algebra has since grown to be a large area of research within algebraic topology \cite{cartan2016homological,Hatcher}.
Theoretical advances, including persistent \cite{EdelsbrunnerLetscherZomorodian2002,zomorodian2005computing} and zigzag \cite{carlsson2010zigzag} homology, began in the early 2000s and formed the field of topological data analysis (TDA), or more generally computational topology \cite{edelsbrunner2010computational,ghrist2014elementary}.
In recent years computational tools have emerged to allow the application of (persistent) homology to real data sets in which calculation by hand would be nearly impossible \cite{otter2017roadmap}. 
TDA has been applied with great success in a variety of application areas including computational biology \cite{bhaskar2023topological,zebrafish, TZH15, wu2017optimal}, neuroscience \cite{anderson2018topological, bendich2016persistent, chung2021reviews, giusti2016two}, geospatial data \cite{feng2021persistent}, computer graphics \cite{pointcloud-topo, singh2007topological}, machining \cite{KHASAWNEH2018195}, and robotics \cite{pathplanning}. These and numerous other applications, many of which can be found in the DONUT database \cite{DONUT}, show that the shape of data is indeed meaningful.

In order to apply homology (and then persistent homology) to a data set, one must derive a structure from the data from which it is possible to compute homology, e.g., a simplicial complex, chain complex, or topological space.
For simplicity we will refer to these structures as ``topological objects'' in this survey. We recognize that chain complexes are not themselves topological objects, but we use the term to mean something that we can compute homology of.
In many cases, there is a straightforward \emph{canonical} way to perform such a transformation.
From a point cloud or metric space, we derive a Vietoris-Rips (VR) or \v{C}ech complex; from a function, we construct a sub- or super-level set.
Data that can be represented as a graph is natively a 1-dimensional simplicial complex, and you can take its clique complex to form a higher dimensional simplicial complex.
One could also form an infinite metric space from a graph called a metric graph using the shortest path metric, whose one-dimensional persistent homology (using the VR or \v{C}ech complex) was characterized in \cite{gasparovic2018complete}.

But some data are too complex to be unambiguously represented by a point cloud, function, or graph.
Take for example academic collaborations.
While it is true that collaboration graphs have provided tremendous value to understand the way people and research topics interact \cite{batagelj2000some,newman2001structure,newman2004best}, these graphs model multi-way collaborations as groups of pairwise collaborations, i.e., graph edges.
But, given a collaboration graph with edges, e.g., $(a,b),$ $(b,c),$ and $(a,c)$, it is not possible to identify whether these edges are modeling three 2-author papers or one 3-author paper without additional information.
Another example comes from biology where proteins can interact in complex ways requiring sometimes many proteins or other enzymes to be present in order for a reaction to occur.
Network biology studies these systems using protein-protein interaction graphs \cite{stelzl2005human}, again modeling what are truly multi-way relationships with groups of pairwise interactions.

On the surface, it may seem that systems with multi-way interactions are closer to a topological space than a point cloud is, since they are collections of subsets of an overall set (e.g., researchers or proteins), and a topological space is a collection of subsets, albeit with some extra properties.
But these collections of sets of researchers or proteins are more general than a topological space and imposing such extra structure requires the addition of multi-way interactions that don't appear in the data, or removal of any that violate the topological space properties.
Instead, it is more accurate to represent these systems as hypergraphs, a higher dimensional analog of a graph.
Given the success TDA has already found in the data science community, and the promise of using topology to make sense of complex data represented by point clouds, functions, or graphs, it seems natural to extend the theory of homology to complex hypergraph-structured data.
However, as opposed to the case of point clouds, functions, or graphs, there is not one canonical way to derive an appropriate topological object if we wish to apply homology to hypergraphs.

The network and data science communities have been moving in the direction of using hypergraphs as data models in recent years \cite{aksoy_hypernetwork_2020,juul2024hypergraph,sharma2020thesis,traversa2023robustness}.
Similarly, many in the TDA community have recognized that hypergraphs can be studied from a topological perspective and have defined several homology theories for hypergraphs (which we describe in detail and cite in Section \ref{sec:homology}). 
However, it is apparent that no canonical solution exists.
Or rather, the straightforward approach of building a simplicial complex by adding all subsets of every hyperedge captures only one of the many notions of ``shape'' or structure found in a hypergraph.

In this paper, we present a survey of recent work to define topological objects from hypergraphs---specifically simplicial, relative, and chain complexes---that can be used to build homology theories for hypergraphs.
We provide a running example showing how to construct each topological object and the resulting homology.
We also discuss some interesting properties of each theory to show the types of hypergraph properties that are captured (or not).
Finally, we compute these homology theories for several small examples to show the variability of the theories and build intuition.
Six of the homology theories we survey have appeared in previous publications and are well-studied; three of them (Sections \ref{sec:ResBS}, \ref{sec:RelBS}, and \ref{sec:weighted}) are new and introduced here.

This paper is organized as follows: In Section~\ref{sec:preliminaries}, we provide background and definitions for simplicial complexes and homology as well as hypergraphs. We define nine different homology theories for hypergraphs in Section~\ref{sec:homology} by describing the transformation of a hypergraph into a topological object of a simplicial, relative, or chain complex on which to compute homology. We also discuss a variety of properties, where known, for these homology theories including functoriality, connections to duality, and recoverability of the hypergraph from the topological structure, highlighting related work and open problems along the way. In Section~\ref{sec:intuition}, we consider several examples to build intuition for each homology method, discussing what happens to the homology when a sub- or super-hyperedge is added to a hypergraph. We conclude with a discussion of interesting research directions in Section~\ref{sec:discussion}.

\paragraph{A note on directed hypergraphs} Just as there is a rich theory around directed graphs and homology defined for directed graphs (see, e.g., \cite{chowdhury_persistent_2018, DOCHTERMANN2023103704, GrigoryanLinMuranov2020}), there is also an active research area for directed hypergraphs including homological notions (e.g., \cite{bubenik2024homotopy,diestel2020homological}).
However, our survey will focus only on undirected hypergraphs as there is much to cover in the undirected case alone.
\section{Preliminaries}
\label{sec:preliminaries}

In this section, we begin by briefly reviewing the essentials needed to study the homology of simplicial complexes and chain complexes. We then shift our focus to several fundamental concepts related to hypergraphs, including a discussion of category-theoretic formulations.

\subsection{Topology and homology}
Topology is the study of invariants of spaces under continuous deformation.
Algebraic topology uses an algebraic language to define the invariants.
Here, we describe homology, a notion from algebraic topology that is concerned with finding ``holes'' in spaces.
We review simplicial homology first, followed by relative homology and chain complex homology.
For more details see \cite{Hatcher}.

\subsubsection{Simplicial complexes}
\label{sec:simplicial}
Simplicial homology captures the homology groups of a simplicial complex.
In this survey, we work with abstract simplicial complexes (though we will continue to simply use the term ``simplicial complex'').
Given a base set $X$, an \emph{(abstract) simplicial complex}, $K = \{\sigma\}$, on $X$ is a collection of nonempty subsets of $X$ with the property that if $\sigma \in K$ and $\tau \subset \sigma$, then $\tau \in K$.
In other words, $K$ is closed under the subset relation.
Each subset of $X$ in $K$ is a \emph{simplex}.
The \emph{dimension} of a simplex is its size minus $1$ and we will say a simplex of dimension $p$ is a $p$-simplex.
Any maximal simplex, not contained in any other simplices, is called a \emph{facet}.
The \emph{dimension} of a simplicial complex is the dimension of its maximum dimensional facet.
Any abstract simplicial complex can be associated with a geometric realization, whose visualization may help the reader understand conceptually what the homology is capturing:
$0$-simplices are vertices, $1$-simplices are edges, $2$-simplices are (filled) triangles, $3$-simplices are (solid) tetrahedra, and so on.

Simplicial complexes form a category $\Simp$ where the objects are simplicial complexes and the morphisms are simplicial maps. Given two simplicial complexes $K_1$ and $K_2$ on base sets $X_1$ and $X_2$ respectively, a map $f: X_1 \rightarrow X_2$ is a \emph{simplicial map}, if for every $\sigma \in K_1$, the set $\{f(x) : x \in \sigma\}$ is a simplex in $K_2$.
As we define various transformations from hypergraphs to simplicial complexes later, we will show that many of them are functors from a category of hypergraphs $\Hyp$ or $\MHyp$ (defined in Section~\ref{sec:categories}) to $\Simp$.

A \emph{weighted simplicial complex} $(K,w)$ is a simplicial complex $K$ together with a weight function $w: K \rightarrow \R$.
Two weighted simplicial complexes, $(K_1, w_1)$ on $X_1$ and $(K_2, w_2)$ on $X_2$, are isomorphic if their underlying structures are isomorphic via a simplicial map $\phi : X_1 \rightarrow X_2$ such that for every $\sigma \in K_1$, $w_2(\{\phi(v) : v \in \sigma\}) = w_1(\sigma)$. In other words, there is an isomorphism of the simplicial complex structures such that the weights on corresponding simplices are equal.

\subsubsection{Simplicial homology}\label{sec:simp_homology}
For the purposes of homology computation, an \emph{orientation} for each simplex needs to be chosen.
The orientation of a $p$-simplex is given by an ordering of its vertices $[x_{0}, \ldots, x_{p}]$.
Transposing two elements in the orientation causes a sign flip (e.g., $[x_1, x_0] = -[x_0, x_1]$).
This choice of orientation can be arbitrary and the final homology calculation is invariant to the orientation (up to isomorphism).

Given a simplicial complex, $K$, a \emph{$p$-chain}, $c$, is a formal sum of oriented $p$-simplices, $\sigma_i$, in $K$ with coefficients, $a_i$, in some group or field $\K$, that is, $c = \sum a_i \sigma_i$ with $a_i \in \K$.
Unless otherwise stated in the remainder of this survey we use $\K=\Z/2\Z$ (called \emph{modulo 2 coefficients}).
The \emph{$p$-chain group}, $\Cgroup_p$, is the group of all $p$-chains where addition is defined component-wise, i.e.,
for $c = \sum a_i \sigma_i$ and $c' = \sum a'_i \sigma_i$, then $c+c'=\sum(a_i+a'_i)\sigma_i$.
When $\K$ is a field $\Cgroup_p$ has the additional structure of a vector space, but it is still commonly referred to as a chain group.
The \emph{boundary} of an oriented $p$-simplex, $\sigma = [x_{0}, \ldots, x_{p}]$, is the alternating\footnote{When $\K=\Z/2Z$ the $(-1)^j$ can be removed since $1=-1$, but we keep it here for the sake of generality of the definition.} sum of its $(p-1)$-dimensional faces.
That is,
\[\bdr_p \sigma = \sum_{j=0}^p (-1)^j[x_{0}, \ldots, \widehat{x}_{j}, \ldots, x_{p}],\]
where $\widehat{x}_{j}$ indicates that the vertex $x_{j}$ is removed, yielding a $(p-1)$-simplex.
The boundary of a $p$-chain can be computed by extending this linearly and thus is a $(p-1)$-chain.
It is straightforward to show that the composition $\bdr_{p-1}\circ \bdr_p = 0$, allowing us to create a sequence of spaces and linear maps,
\[ \cdots \Cgroup_p \stackrel{\bdr_p}{\longrightarrow} \Cgroup_{p-1} \cdots \stackrel{\bdr_2}{\longrightarrow} \Cgroup_1 \stackrel{\bdr_1}{\longrightarrow} \Cgroup_0 \stackrel{\bdr_0}{\longrightarrow} 0, \]
such that $\ker\bdr_{p-1} \subseteq \im\bdr_{p}$.
A sequence with this property is called a \emph{chain complex}.

The kernel of $\bdr_{p}$ contains all $p$-chains whose boundary is zero.
To gain intuition, consider the simplicial complex $K=\{a,b,c,ab,ac,bc\}$ (using shorthand where $ab$ means the 1-simplex $[a,b]$).
The boundary of the 1-chain $[a,b]+[b,c]+[c,a]$ is
\[ (b-a) + (c-b) + (a-c) = 0.\]
This 1-chain represents the boundary of a triangle which has no 0-dimensional ``endpoints,'' aligning with the fact that the boundary is 0.
We refer to $\Zgroup_{p} := \ker\bdr_{p}$ as the \emph{group of $p$-cycles}.
Then, the image of $\bdr_{p+1}$, by definition, are all of the boundaries of $(p+1)$-chains.
For example, the boundary of the 2-chain $[a, b, c]$ is $[b,c] - [a,c] + [a,b]$.
We refer to $\Bgroup_p := \im\bdr_{p+1}$ as the \emph{group of $p$-boundaries}.
The \emph{$p$-th simplicial homology group} is the quotient $\Hgroup_p = \Zgroup_{p}/\Bgroup_p$.
Intuitively, these are the cycles in dimension $p$ that are not the boundary of a collection of $(p+1)$-dimensional simplices in $K$.
The dimension\footnote{``Dimension'' refers to the rank or dimension of $\Hgroup_p$, not the dimension $p$ of the simplices in $\Cgroup_p$.} of the $p$-th homology group is called the \emph{$p$-th Betti number}, denoted $\betti{}{p}$.
To denote homology and Betti number of an arbitrary dimension, we use $\hom{}{\bullet}$ and $\betti{}{\bullet}$, respectively.
Two $p$-chains, $\sigma, \kappa \in \Cgroup_p$, are said to be \emph{homologous} if their difference is in $\Bgroup_p$, i.e., $\sigma - \kappa$ is the boundary of a $(p+1)$-chain.

\subsubsection{Relative simplicial homology}\label{sec:rel_hom}
Simplicial homology, sometimes referred to as \emph{absolute} simplicial homology, can be extended analogously to relative (simplicial) homology, which computes the homology of a simplicial complex $K$ relative to a subcomplex $K_0 \subseteq K$. Intuitively, relative homology computes the reduced homology in the complement $K \setminus K_0$ when the subcomplex is identified to a point.
The \emph{relative $p$-chain group} $\Cgroup_p(K,K_0):=\Cgroup_p(K)/\Cgroup_p(K_0)$ is a quotient of the chain groups.
The boundary map $\bdr_p$ induces a quotient boundary map $\bdr_p:\Cgroup_p(K,K_0)\to \Cgroup_{p-1}(K,K_0)$ since $\bdr_p$ takes the $p$-chains of the subcomplex $K_0$ to $(p-1)$-chains of the subcomplex.
We can thus form \emph{relative $p$-cycle groups} $\Zgroup_p(K,K_0):=\ker \bdr_p$, \emph{relative $p$-boundary groups} $\Bgroup_p(K,K_0):=\im \bdr_{p+1}$, and a \emph{relative chain complex} with $\bdr_{p-1}\circ\bdr_p = 0$, analogously.
The \emph{relative homology group} is defined to be  $\Hgroup_p(K,K_0):= \Zgroup_p(K,K_0)/\Bgroup_p(K,K_0)$.

As shown in~\cite[pp. 124-125]{Hatcher}, relative homology can be expressed as reduced absolute homology by considering the space $K \cup CK_0$, where $CK_0$ is the cone $(K_0 \times I)/(K_0 \times \{0\})$ whose base $K_0 \times \{1\}$ we identify with $K_0 \subseteq K$.
That is, $\Hgroup_p(K, K_0)$ is isomorphic to $\tilde{\Hgroup}_p(K \cup CK_0)$.

\subsubsection{Homology given a chain complex}
\label{sec:chain-homology}
The notion of a chain complex defined above for simplicial homology can be made much more general.
Any sequence of vector spaces, $\{\Cgroup_p\}$, and linear maps, $\bdr_p : \Cgroup_p \rightarrow \Cgroup_{p-1}$, with the property that $\bdr_{p-1}\circ\bdr_{p} = 0$ for all $p$ is a chain complex.
The subsequent definitions of $\Zgroup_p$, $\Bgroup_p$, and $\Hgroup_p = \Zgroup_{p}/\Bgroup_p$ all follow identically.
Because there may not be a simplicial complex underlying the chain complex,  we may lose intuition of homology as cycles that are not boundaries, but the computation is valid. We call $\Hgroup_p$ the $p$-th homology of the chain complex, with $\Hgroup_{\bullet}$ denoting the homology across all dimensions.

\subsection{Hypergraphs}
In this section, we define the concept of a hypergraph as well as properties and related structures that are relevant to our topological exploration of hypergraphs. Many of these definitions are standard as found in references such as \cite{berge1984hypergraphs,bretto2013hypergraph,aksoy_hypernetwork_2020}.

\subsubsection{Hypergraph basics}
In Section \ref{sec:categories} we will give a category theoretical definition of hypergraphs and hypergraph morphisms. Here we give a standard definition of hypergraphs that we will use throughout the survey unless we are making a category theory argument. A \emph{hypergraph}, $\H=(V, E, \epsilon)$, is a finite set of \emph{vertices}, $V$, together with a finite set of \emph{hyperedges}, $E=\{e_1, \ldots, e_m\}$, and a function $\epsilon : E\rightarrow 2^V$ that identifies the vertices in hyperedge $e_i$ as $\epsilon(e_i)$.
We think of $E$ as the names of the hyperedges and $\epsilon$ as identifying which vertices are in each hyperedge.
A hypergraph in which all hyperedges have size $k$ is called a \emph{$k$-uniform} hypergraph.
Hypergraphs generalize graphs. Indeed, every graph is a 2-uniform hypergraph.
When it is clear from context, we use \emph{edges} and \emph{hyperedges} interchangeably.

We note that this hypergraph definition with a mapping from names of hyperedges to sets of vertices does not forbid multi-edges, where $\epsilon(e) = \epsilon(e')$ for some $e \neq e'$, nor does it forbid the empty set as a hyperedge. Multi-edges are quite common in real hypergraph data, and empty edges are also possible. However, in some of the homology theories below (e.g., closure homology, embedded homology), multi-edges and empty edges are ignored. In those cases we will assume that multi-edges have been ``collapsed'' into single edges and empty edges have been removed.
We will point out when this is required and when multi- and empty edges are not ignored or problematic.

For ease of exposition, in the remainder of the paper we will often use simplified notation $\H=(V, E)$ and write that for each $e \in E$ we have $e \subset V$. We may also write $v \in e$ to mean that $v \in \epsilon(e)$.

\subsubsection{Hyperblock}
A simplicial complex is a special type of hypergraph.
Given an arbitrary hypergraph, $\H$, we can form its associated simplicial complex or its \emph{upper closure}, by first collapsing multi-edges, removing empty edges, and then adding all nonempty subsets of each remaining hyperedge, thus forming a simplicial complex. We denote this by $\Delta(\H)$.
Similarly, we can create a hypergraph with no containment among hyperedges by removing any hyperedge that is contained in a larger hyperedge.
In general, a hypergraph with no edge containment relations is called a \emph{simple} hypergraph\footnote{In graphs the term ``simple'' is used to mean no loops and no multi-edges. However, in the hypergraph literature the term is used for the stronger property of no edge containment relations \cite{berge1984hypergraphs,joslyn2020hypernetwork}. This implies that there are no multi-edges and that the only ``loops'' (singleton edges) must be isolated, i.e., not contained within any larger edges.}.
Any hyperedge that is not contained in any other hyperedge is a ``top level'' edge, or a \emph{toplex}.
In a simple hypergraph, all edges are toplexes.
Finally, there is a collection of hypergraphs that all have the same upper closure and simple hypergraph, with varying levels of inclusions among hyperedges.
Such a collection of hypergraphs is called a \emph{hyperblock} \cite[page 5]{joslyn2020hypernetwork}, which will come up in our survey of homology methods for hypergraphs (Section~\ref{sec:homology}).
We will observe that, for some homology theories, the homology may be the same across all hypergraphs in a hyperblock, and in other cases, it may vary across a hyperblock.
Structurally, every hypergraph can be considered as being nested between two hypergraphs in a hyperblock, its upper closure and its simple hypergraph.

Another way of describing $\Delta(\H)$ is that it is the smallest simplicial complex that contains $\H$ as a subset (after removing multi-edges and empty edges from $\H$). We can flip this around and introduce $\delta(\H)$ as the largest simplicial complex contained within $\H$. Unless $\H = \Delta(\H)$, this \emph{lower closure} is not in the same hyperblock as $\H$. In practice in hypergraphs observed from real-world datasets, $\delta(\H)$ tends to be much smaller than $\H$ (possibly empty) and may not be informative of the hypergraph structure. However, it is a valid simplicial complex and provides insight into what portion of the hypergraph is a simplicial complex. Intuitively, $\Delta(\H)$ is the simplicial complex obtained from $\H$ via ``additions," whereas $\delta(\H)$ is obtained from $\H$ via ``peeling."

\subsubsection{Incidence matrix and duality}
A hypergraph can be represented by its \emph{incidence matrix}, $S$, a binary matrix with $|V|$ rows and $|E|$ columns in which there is a $1$ in row $i$, column $j$ if and only if vertex $v_i$ is in edge $e_j$ and all other entries are $0$.
If data is provided in the form of a $0$-$1$ matrix, a hypergraph can be unambiguously constructed by considering the matrix as its incidence matrix.
The transpose, $S^T$, of a given incidence matrix, $S$, is also a 0-1 matrix that can be represented by a hypergraph.
The two hypergraphs, created from $S$ and $S^T$, are \emph{dual} hypergraphs formed from the same incidence relation.
The dual of a hypergraph can be formed without passing through the incidence matrix by swapping the roles of vertices and hyperedges.
Formally, the dual $\H^*=(E^*, V^*)$ of hypergraph $\H=(V, E)$ has vertices $E^* = \{e_1^*, \ldots, e_m^*\}$ and edges $V^* = \{v_1^*, \ldots, v_n^*\}$ such that $v_i^* = \{e_j^* : v_i \in e_j \text{ in }\H\}$.
If $\H$ has multi-edges, all copies of the edges are in $E^*$. Any empty edges in $E$ will be isolated vertices in $\H^*$.

\subsubsection{Line graphs and nerves}
A line graph is a graph associated to a hypergraph that captures the pairwise intersection relations among the hyperedges. Hyperedges of $\H$ correspond to vertices of $L(\H)$, and edges in $L(\H)$ reflect nonempty intersections  among pairs of hyperedges in $\H$.

\begin{definition}
The \emph{line graph} $L(\H)$ of a hypergraph $\H$  consists of a vertex set $\{e^*_1, \cdots, e^*_m\}$, and an edge set $\{(e^*_i, e^*_j) \mid e_i \cap e_j \neq \emptyset, i \neq j\}$.
\end{definition}

\begin{definition}\label{def:nerve}
Let $\H=(V, E)$ be a hypergraph. The \emph{nerve} of $\H$, denoted $\Nrv(\H)$, is a simplicial complex on the base set $E$ such that $\sigma \in \Nrv(\H)$ whenever $\cap_{e \in \sigma} e \neq \emptyset$. In other words, there is a simplex in $\Nrv(\H)$ for every set of hyperedges with nonempty intersection.
\end{definition}

Intuitively, hyperedges form a cover of the vertices, and a line graph can be considered as the $1$-dimensional nerve of the cover.
Just as in the dual, all multi and empty edges in $\H$ are vertices in $L(\H)$ and $\Nrv(\H)$.

\subsubsection{Walks, cycles, and components}
We now describe the concepts of hypergraph walks and components, as introduced in \cite[Def. 5]{aksoy_hypernetwork_2020}.
Two hyperedges $e, f \in E$ are \emph{$s$-adjacent} if $|e \cap f| \geq s$, i.e., they intersect in at least $s$ vertices.
Then, we say that $e=e_{0}, e_{1}, \ldots, e_{k}=f$ is an \emph{$s$-walk} of length $k$ from $e$ to $f$, if $e_{j}$ and $e_{j+1}$ are $s$-adjacent for $0 \leq j\leq k-1$.
If $e=f$ then the $s$-walk is \emph{closed}.
Related notions of $s$-trace, $s$-meander, $s$-path, and $s$-cycle are defined in \cite{aksoy_hypernetwork_2020} to generalize paths and cycles in graphs.
Other notions of acyclicity appear in \cite{berge1984hypergraphs} and \cite{fagin1983degrees}.
A set of hyperedges $F \subseteq E$ is $s$-connected if there is an $s$-path between any pair $f_i, f_j \in F$.
$F$ is an \emph{$s$-component} (or \emph{$s$-connected component}) if it is $s$-connected and there is no $s$-connected $F' \subset E$ such that $F \subsetneq F'$.

\subsubsection{Categories of hypergraphs}
\label{sec:categories}

The main focus of this paper is to survey homology theories for hypergraphs, not to provide a deep category theoretical treatment of hypergraphs and homology.
However, homology (particularly homological algebra) has been closely related to category theory since the inception of category theory. For example, it is well-known that simplicial homology is a functor from the category $\Simp$ to the category of abelian groups (or vector spaces if using field coefficients), denoted as $\Ab$. The functoriality of simplicial homology enables the theory of persistent homology \cite{EdelsbrunnerLetscherZomorodian2002}, as simplicial maps induce homomorphisms on the homology groups through the homology functor.
Persistent (simplicial) homology applied to point clouds and functions has shown incredible value to data science (e.g.,~\cite{CarlssonVejdemo-Johansson2021,YanMasoodSridharamurthy2021}) and we expect that the notions of homology and persistent homology for hypergraphs would have similar value.
Indeed, we have already seen the practical application of the barycentric homologies (defined in Sections~\ref{sec:ResBS} and \ref{sec:RelBS}) in classification~\cite{aktas2023hypergraph}.
Therefore, to enable proofs and discussions of functoriality throughout Section \ref{sec:homology} we will introduce and discuss some hypergraph categories here.

We define two categories of hypergraphs following
\cite{dorfler1980category,grilliette2023incidence,grilliette2020simplification,green2025topological}: $\MHyp$ is the category of multi-hypergraphs and $\Hyp$ is the category of hypergraphs without multi-edges.
To best mirror the definition of a hypergraph above we will use slightly different notation than those papers. We will use the category names from \cite{green2025topological}.
We assume the reader is familiar with the category of sets and set maps, denoted as $\Set$, and the full subcategory of finite sets, denoted as $\FinSet$.
Although it is possible that many of the results in this paper would apply if hypergraphs have infinite vertex sets, infinite hyperedges, or infinitely many hyperedges, in this paper we will restrict our hypergraph categories to finite hypergraphs for simplicity and to align with real world hypergraphs.

We first define $\MHyp$, a category that dates back to \cite{dorfler1980category}. In \cite{grilliette2023incidence}, the authors remove D\"{o}rfler and Waller's requirement that hyperedges be nonempty and then formulate it as a comma category, denoting it $\mathfrak{H}=(id_\Set \downarrow \po)$, relying on the covariant powerset functor.

\begin{definition}[{\cite[page 13]{mac2013categories}}]
    The \emph{covariant powerset functor}, $\po: \Set \rightarrow \Set$ is defined as follows:
    \begin{itemize}
        \item For an object $X \in \Set$, $\po(X)$ is the powerset of $X$.
        \item For a morphism $\phi : X \rightarrow Y$, $\po(\phi)(A) := \{\phi(x) : x \in A\}$, the image of $A$ under $\phi$, where $A \subseteq X$.
    \end{itemize}
\end{definition}

For readers not familiar with the notion of a comma category, the objects and morphisms are as follows:
\begin{definition}[{\cite[page 6]{grilliette2023incidence}}]
    The category of (finite) hypergraphs, $\MHyp$, is given by the following objects and morphisms:
    \begin{itemize}
        \item $Ob(\MHyp) = \{ \H = (V, E, \epsilon) : V, E \in \FinSet, \epsilon : E \rightarrow \po(V)\}$, i.e., an object $\H$ of $\MHyp$ consists of a finite set of vertices, $V$, a finite set of hyperedges, $E$, and a function, $\epsilon$, that maps each hyperedge to its corresponding set of vertices.
        \item A morphism $(V, E, \epsilon) \rightarrow (V', E', \epsilon')$ is a pair $(f_V, f_E)$ where $f_V : V\rightarrow V'$, $f_E : E\rightarrow E'$ are morphisms in $\FinSet$ such that the following diagram commutes:
            \[ \begin{tikzcd}
            E \arrow{r}{f_E} \arrow[swap]{d}{\epsilon} & E' \arrow{d}{\epsilon'} \\%
            \po(V) \arrow{r}{\po(f_V)}& \po(V')
            \end{tikzcd}
            \]
        \item Composition of morphisms is component-wise: \[(f_V, f_E) \circ (f_V', f_E') = (f_V\circ f_V', f_E\circ f_E')\]
    \end{itemize}
\end{definition}
In words, a morphism from hypergraph $\H$ to hypergraph $\H'$ in $\MHyp$ requires a map between vertices ($f_V$) and a map between hyperedges ($f_E$) such that the set of vertices of the hyperedge that $e\in E$ maps to ($\epsilon'(f_E(e))$) is equal to the set of vertices that are mapped to by the vertices in $e$ ($\po(f_V)(\epsilon(e))$).
We reiterate that unless we are making a category theoretical argument, we will refer to $\epsilon(e)$ as simply $e$. For instance, instead of saying $v \in \epsilon(e)$, we will simply say $v \in e$.

While $\MHyp$ aligns best with real hypergraph data that frequently includes multi-edges, we remark that many of the homology theories we will survey ignore multi-edges.
Therefore, we also introduce the category $\Hyp$ of hypergraphs without multi-edges and will restrict many of the category theoretical proofs in Section \ref{sec:homology} to $\Hyp$.

\begin{definition}[\cite{green2025topological}, page 7 with Def. 14; \cite{grilliette2020simplification}, page 9 as $\SSys$]
The category of (finite) hypergraphs without multi-edges, $\Hyp$, is given by the following:
\begin{itemize}
    \item $Ob(\Hyp) = \{\H = (V, E) : V \in \FinSet, E \subseteq \po(V)\}$
    \item A morphism $f$ from $\H=(V, E)$ to $\H'=(V', E')$ is a function $f : V\rightarrow V'$ such that for all $e \in E$ the set $\{f(v)\}_{v \in e}$ is an edge in $E'$.
\end{itemize}
The fact that composition of morphisms is a morphism is proven in \cite[Lemma 1]{green2025topological}. Associativity and the identity map are inherited from $\FinSet$.
\end{definition}

In \cite[Sec. 2.1.2]{grilliette2024simplification} and restated in \cite[Lemma 3]{green2025topological}, the authors remark that there is an embedding functor from $\Hyp$ into $\MHyp$ whose left adjoint is the collapse functor from $\MHyp$ to $\Hyp$ \cite[page 9]{grilliette2020simplification}.
On objects, the embedding functor ($Emb$) takes a hypergraph $(V, E)$ with $E \subseteq \po(V)$ to the same vertex and edge set, with incidence function $\epsilon$ as the identity since every $e \in E$ is already in $\po(V)$.
The collapse functor ($Col$) takes a multi-hypergraph $(V, E, \epsilon)$ to a hypergraph without multi-edges, $(V, E_c)$, where $E_c \subseteq \po(V)$ and $\hat{e}$ is in $E_c$ if there is some $e \in E$ with $\epsilon(e) = \hat{e}$.
For sake of completeness we formally restate the functors, including the morphism maps, here in our notation.
\begin{align*}
    Emb : \Hyp& \rightarrow \MHyp\\
        (V, E)& \mapsto (V, E, \epsilon(e) = e) \hspace{2em} \textrm{(object map)}\\
        f& \mapsto (f, \po(f)) \hspace{4em} \textrm{(morphism map)}\\
        \\
    Col : \MHyp& \rightarrow \Hyp\\
        (V, E, \epsilon)& \mapsto (V, E_c) \hspace{2em} \textrm{(object map)}\\
        (f_V, f_E)& \mapsto f_V \hspace{3em} \textrm{(morphism map)}
\end{align*}

In fact, the $Emb$ map makes $\Hyp$ a \emph{full} subcategory of $\MHyp$. This is not proven in any of these references but the proof is straightforward. Given two hypergraphs in $\Hyp$ brought over to $\MHyp$, $\H = Emb(V, E)$ and $\H' = Emb(V',E')$ one can show that if $(f_V, f_E)$ and $(f_V, f_E')$ are two morphisms from $\H$ to $\H'$ with identical vertex maps the fact that there are no multi-edges in $\H$ or $\H'$ implies that $f_E = f_E'$. Therefore, the only morphism between them with $f_V$ as the vertex map must be $(f_V, \po(f_V))$.

We also remark that $\Simp$ is a full subcategory of $\Hyp$ by definition. Objects in $\Simp$ are simplicial complexes $K$ on base sets $X$ (analogous to $E$ and $V$ respectively).
Objects in $\Simp$ are more restrictive since the simplices (edges) have the downward closure requirement.
A simplicial map in $\Simp$ is a map on vertices which takes simplices to simplices, just like morphisms in $\Hyp$ are vertex maps that must take hyperedges to hyperedges. In Figure \ref{fig:hyp_cat_diagram}, we summarize how these three categories, $\Simp$, $\Hyp$, and $\MHyp$ are all related.
\begin{figure}[h]
    \centering
    \includegraphics[width=0.5\linewidth]{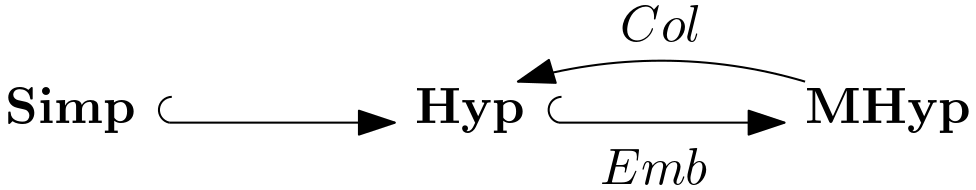}
    \caption{A diagram showing relationship between categories $\Simp$, $\Hyp$, and $\MHyp$ with $\hookrightarrow$ indicating full subcategory.}
    \label{fig:hyp_cat_diagram}
\end{figure}

Although the duality operation takes an object from $\MHyp$ to another object in the same category the reader may notice that duality does not appear in Figure~\ref{fig:hyp_cat_diagram}. This is because it was shown in \cite[Prop. 2.15]{grilliette2024simplification} that duality is not a functor on $\Hyp$ or $\MHyp$.
Indeed, one can create a counterexample of a pair of hypergraphs with a morphism between them in $\MHyp$ for which their duals have no morphims between them.

While both $\MHyp$ and $\Hyp$ allow the empty set to be a hyperedge, objects in $\Simp$ are not allowed to contain the empty set as a simplex.
Because the empty edge is a special case we say a little more here about its role in $\Hyp$ and $\MHyp$.
Let $f$ be a morphism in $\Hyp$ from $\H_1$ to $\H_2$.
Since $f$ is a morphism in $\Hyp$ it must take the vertices of hyperedges in $\H_1$ to the vertices of hyperedges in $\H_2$. Therefore, if $\emptyset$ is an edge in $\H_1$ then $\emptyset$ must be a hyperedge in $\H_2$.
The same is true for $\MHyp$, though you must go through the commutative diagram to show it.
In both $\Hyp$ and $\MHyp$, hypergraphs without empty edges can map to hypergraphs with empty edges since not every edge has to be mapped to.
Some of the homology theories in Section \ref{sec:homology} remove empty edges in the process and we will point out where this becomes an issue with functoriality.

Finally, we note that there are other categories of hypergraphs, for instance in the setting of metric measure spaces~\cite{ChowdhuryNeedhamSemrad2023} and incidence hypergraphs \cite[Definition 3.11]{grilliette2023incidence}. However, since our interest in including a category theoretical definition of hypergraphs is to prove that various notions of homology for hypergraphs are functors to $\Ab$, the fact that $\Simp$ is a full subcategory of $\Hyp$ draws clear parallels to the case of simplicial homology and motivates our use of $\Hyp$ as the main category of hypergraphs in this paper. In some cases, we may remark about how a homology theory applies to $\MHyp$ but we will mostly restrict our proofs to the simpler case of $\Hyp$.



\section{Homology Theories for Hypergraphs}
\label{sec:homology}
As a combinatorial object, a hypergraph is not a topological space or a chain complex. It is also not necessarily a simplicial complex (except in special cases). However, we may transform a hypergraph into a topological space, a simplicial complex, or a chain complex in a variety of ways. Each of these transformations yields a homology theory for hypergraphs. In this survey, we focus on the simplicial and chain complex homology rather than the singular homology of a hypergraph, by transforming a hypergraph into a simplicial complex or a chain complex rather than an arbitrary topological space. 
To that end, we describe and discuss the following homology theories for hypergraphs in detail including results and examples:
\begin{itemize}
\item {Closure homology} (Section~\ref{sec:closure-homology});
\item {Restricted barycentric subdivision homology} (Section~\ref{sec:ResBS});
\item {Relative barycentric subdivision homology} (Section~\ref{sec:RelBS});
\item {Polar complex homology} (Section~\ref{sec:polar}); 
\item {Embedded homology} (Section~\ref{sec:embedded});
\item {Path homology} (Section~\ref{sec:path-homology});
\item {Weighted nerve complex \emph{persistent} homology} (Section~\ref{sec:weighted}).
\end{itemize} 
We also include a short section on chromatic and magnitude homology (Section \ref{sec:mag_chrom}), two homology theories that are well-studied for graphs but more nascent for hypergraphs, and a short section with references to other topological treatments of hypergraphs that are not simplicial, relative, or chain complex based.
Our survey includes known (with citations) and new (without citations) results; theorems with citations are occasionally proved using our notations or different methods for completeness.  

\subsection{Simplicial homology of hypergraph upper closure}
\label{sec:closure-homology}

We begin our survey with the homology of what we feel is the simplest transformation from a hypergraph to a simplicial complex, the upper closure.
Following the preliminaries in Section~\ref{sec:preliminaries}, we can give this definition immediately.
\begin{definition}
Let $\H=(V, E)$ be a hypergraph. 
The \emph{closure homology} of $\H$, denoted $\hom{\Delta}{\bullet}(\H)$, is the homology of its upper closure $\Delta(\H)$, i.e.,  $\hom{\Delta}{\bullet}(\H) := \hom{}{\bullet}(\Delta(\H))$.
\end{definition}
Parks and Lipscomb~\cite{upscomb1991homology} considered the closure homology of hypergraphs and showed its relation to different notions of hypergraph acyclicity that appear in \cite{fagin1983degrees}. 
There are a few properties of $\hom{\Delta}{\bullet}$ that we point out.

\begin{proposition}
If $\H$ and $\H'$ are in the same hyperblock, then their closure homologies are equal. 
\end{proposition}
\begin{proof}
This is immediate from the definitions of hyperblock and closure homology, since $\H$ and $\H'$ are in the same hyperblock if and only if their upper closures $\Delta(\H)$ and $\Delta(\H')$ are equal, which in turn implies $\hom{\Delta}{\bullet}(\H) = \hom{\Delta}{\bullet}(\H')$.
\end{proof}

\begin{proposition}
For $\H \in \MHyp$ the closure homology of $\H$ is isomorphic to the closure homology of its dual $\H^*$, i.e., $\hom{\Delta}{\bullet}(\H) \cong \hom{\Delta}{\bullet}(\H^*)$. 
\end{proposition}
\begin{proof}
To show this we turn to the Dowker duality theorem \cite[Theorem 1a]{dowker1952homology}, \cite[Corollary 3]{chowdhury2018functorial}.
This duality theorem is typically stated in terms of binary relations rather than hypergraphs, so we will provide the statement and its connection to hypergraphs.
Let $V$ and $E$ be two totally ordered sets and $S \subseteq V \times E$ be a nonempty relation. 
Note that we can interpret $S$ as the incidence matrix of a hypergraph $\H$, where $(v,e) \in S$ corresponds to a 1 in row $v$ and column $e$ and all other entries are 0.
We can define two simplicial complexes $V_S$ and $E_S$ from $S$ as follows. 
A simplex $\sigma \subset V$ is in  $V_S$ whenever there is an $e \in E$ such that $(v, e) \in S$ for all $v \in \sigma$. 
From the perspective of the incidence matrix, there is a column corresponding to an element $e$ such that the rows corresponding to elements of $\sigma$ are all 1 in that column (and there could be more 1s in that column).
On the other hand, a simplex $\tau \subset E$ is in $E_S$ whenever there is a $v \in V$ such that $(v,e) \in S$ for all $e \in \tau$.
There is a similar incidence matrix interpretation for $E_S$.
One can show that $V_S = \Delta(\H)$ and $E_S = \Delta(\H^*)$. 
The Dowker theorem states that $\hom{}{\bullet}(V_S) \cong \hom{}{\bullet}(E_S)$. 
It then follows directly that $\hom{\Delta}{\bullet}(\H) \cong \hom{\Delta}{\bullet}(\H^*)$.
\end{proof}

Next, we show that $\hom{\Delta}{\bullet} : \MHyp \rightarrow \Ab$ is a functor. 
As noted in Section \ref{sec:categories}, it is well known that simplicial homology is a functor from $\Simp$ to $\Ab$. 
Therefore, we will prove that $\Delta$ is a functor from $\MHyp$ to $\Simp$ and the two together give us the result that $\hom{\Delta}{\bullet}$ is a functor from $\MHyp$ to $\Ab$.

\begin{proposition}\label{prop:closure_functor}
The hypergraph upper closure operation is a functor $\Delta : \MHyp \rightarrow \Simp$.
\end{proposition}
\begin{proof}
The first observation is that $\Delta$ is a composition of the collapse functor, $Col:\MHyp\rightarrow\Hyp$, and the process of adding in sub-edges, which we have not yet formally defined as an object map and morphism map. We do so now for $\Hyp$ to $\Simp$ and prove that it is a functor. Then its composition with $Col$ will also be a functor.
\begin{samepage}
\begin{align*}
    \Delta : \Hyp& \rightarrow \Simp\\
        (V, E)& \mapsto (V, \{\emptyset \neq \sigma \subseteq V : \exists e \in E \text{ with }\sigma \subseteq e\}) \hspace{2em} \textrm{(object map)}\\
        f& \mapsto f \hspace{17em} \textrm{(morphism map)}
\end{align*}
\end{samepage}

The fact that $\Delta(Id_\Hyp) = Id_\Simp$ and $\Delta(f_1\circ f_2) = \Delta(f_1)\circ\Delta(f_2)$ are true by definition since the morphism map is the identity. To show $\Delta$ is a functor it is left to show that if $f$ is a morphism in $\Hyp$ from $\H_1=(V_1, E_1)$ to $\H_2=(V_2, E_2)$ then $\Delta(f)$ is a simplicial map from $\Delta(\H_1)$ to $\Delta(\H_2)$. 

Let $\sigma \in \Delta(\H_1)$, then there is an $e \in E_1$ such that $\sigma \subseteq e$. Since $f$ is a morphism in $\Hyp$ it must be that $\{f(v)\}_{v \in e}$ is an edge in $\H_2$. And of course $\{f(v)\}_{v \in \sigma}$ is a subset of this edge.
Note that since $\{f(v)\}_{v \in e}$ is an edge in $\H_2$ it is also a simplex in $\Delta(\H_2)$. Finally since any subset of a simplex is also a simplex, namely $\{f(v)\}_{v \in \sigma}$, $f$ is a simplicial map and we have the desired result.

    %
\end{proof}

We note that in \cite[Theorem 3]{robinson2022cosheaf} Robinson proved a related result replacing $\MHyp$ with the category of binary relations, $\Rel$. It is not difficult to show that $\MHyp$ is a subcategory of $\Rel$ (e.g., see \cite{green2025topological}). Therefore, Robinson's result is a stronger statement with Proposition \ref{prop:closure_functor} as a corollary.

Finally, we provide a connection between $\hom{\Delta}{\bullet}$ and the homology of the nerve of the hypergraph. 

\begin{proposition}
For $\H \in \MHyp$ the closure homology of $\H$ is isomorphic to the homology of its nerve, i.e., 
$\hom{\Delta}{\bullet}(\H) \cong \hom{}{\bullet}(\Nrv(\H))$.
\end{proposition}
\begin{proof}
The proof of this proposition follows immediately from the claim that $\Nrv(\H) = \Delta(\H^*)$ and Dowker duality. 
Indeed if this claim is true then
\[ \hom{}{\bullet}(\Nrv(\H)) \stackrel{\text{(claim)}}{=} \hom{}{\bullet}(\Delta(\H^*)) \stackrel{\text{(def)}}{=} \hom{\Delta}{\bullet}(\H^*) \stackrel{\text{(Dowker)}}{\cong} \hom{\Delta}{\bullet}(\H). \]
To show $\Nrv(\H) = \Delta(\H^*)$, we argue mutual containment.
Let $\sigma \in \Nrv(\H)$, then by definition $\cap_{e \in \sigma} e \neq \emptyset$.
Let $v \in \cap_{e \in \sigma} e$, then there is a hyperedge in $\H^*=(E^*, V^*)$ corresponding to $v$ that contains all $e \in \sigma$ (possibly more).
Therefore $\sigma \in \Delta(\H^*)$, since $\sigma$ is a subset of a hyperedge in $\H^*$, and so $\Nrv(\H) \subseteq \Delta(\H^*)$.
Then, let $\tau \in \Delta(\H^*)$. 
By definition of the hypergraph upper closure, $\tau$ is a subset of a hyperedge in $\H^*$. 
Therefore, $\tau = \{e^*\} \subseteq v^*$ represents a set of hyperedges in $\H$ that all contain a vertex $v$.
This means that $\cap_{e^* \in \tau} e \neq \emptyset$ and so $\tau \in \Nrv(\H)$. 
This proves the reverse inclusion and therefore $\Nrv(\H) = \Delta(\H^*)$.
Note that this argument holds even if there are multi-edges in $\H$ or $\H^*$.
\end{proof}

A summary of the properties of $\hom{\Delta}{\bullet}$ is given in Figure \ref{fig:closure-homology}. 
Note that this is not a functorial diagram, it is on the level of the specific objects $\H$, $\H^*$, and their nerves, upper closures, and homology groups.
\begin{figure}[htb]
    \centering
    \includegraphics[width=0.7\textwidth]{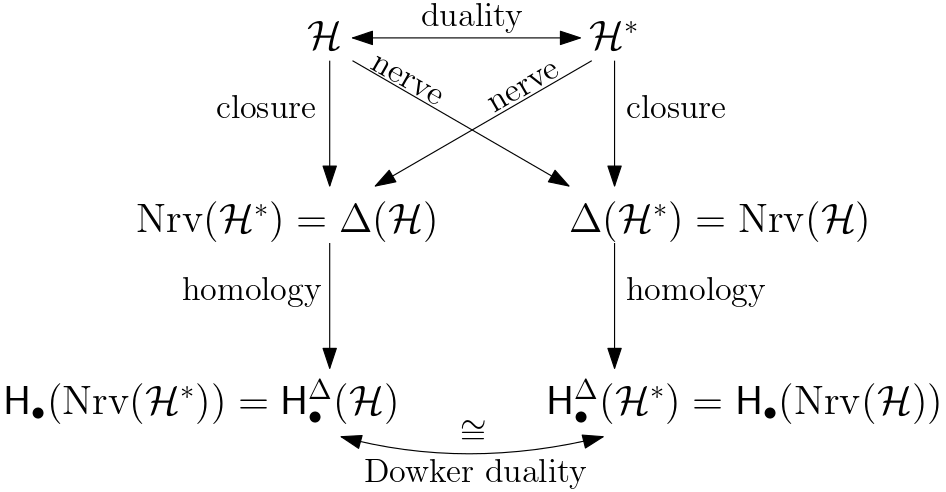}
    \caption{Summary of relationships between a hypergraph and its dual, their upper closures, nerves, and closure homology.}
    \label{fig:closure-homology}
\end{figure}
Similarly, we could also define and study the homology of the lower closure as $\hom{\delta}{\bullet}(\H) = \hom{}{\bullet}(\delta(\H))$, whose details are omitted here.

\subsection{Restricted barycentric subdivision homology}
\label{sec:ResBS}

In this subsection and the next, we introduce two new hypergraph homology theories originally defined by the authors. We first build a construction called the restricted barycentric subdivision (ResBS) of a hypergraph. 

\begin{definition}
\label{def:rbs}
Let $K=\{\sigma_i\}$ be a simplicial complex. 
Its \emph{barycentric subdivision}, $\bary(K)$, is a simplicial complex constructed with the set of simplices of $K$, $\{\sigma_i\}$, as the vertex set. 
There is a $k$-simplex $\{\sigma_{i_0}, \sigma_{i_1}, \ldots, \sigma_{i_k}\}$ in $\bary(K)$ whenever $\sigma_{i_0}\subset\sigma_{i_1}\subset \cdots \subset \sigma_{i_k}$.
The \emph{barycentric subdivision of a hypergraph} $\H$ is the barycentric subdivision of its upper closure, $\bary(\Delta(\H))$.
\end{definition}

An equivalent way to construct $\bary(K)$ is to build a graph using the simplices of $K$ as the vertex set, with an edge $(\sigma_i, \sigma_j)$ whenever $\sigma_i \subset \sigma_j$, and then form the \emph{clique complex} of this graph by replacing every $k$-clique with a $(k-1)$-simplex.

\begin{definition}
The \emph{restricted barycentric subdivision (ResBS)}, $\bary_{res}(\H)$, of a hypergraph $\H$ is the subcomplex of $\bary(\Delta(\H))$ induced by the vertices corresponding to hyperedges of $\H$. 
\end{definition}
\noindent We show examples of these objects for a single hypergraph in Figure~\ref{fig:rest_bary_ex}.
This hypergraph will serve as a running example to illustrate the various homology theories in subsequent subsections.

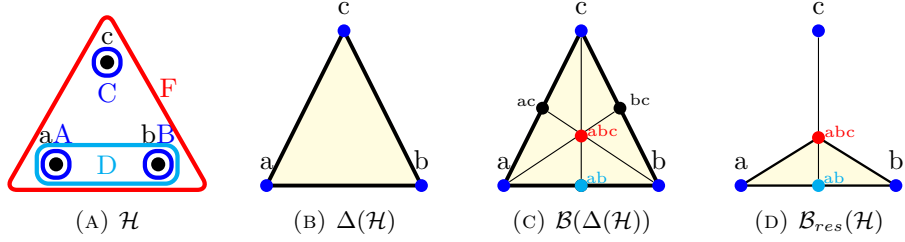
\begin{figure}[ht]
\begin{subfigure}{.24\textwidth}
  \centering
  \resizebox{1.00\textwidth}{!}{
  \begin{tikzpicture}[scale=0.6, ve/.style={fill=black,circle,scale=0.5},  he1/.style={draw = blue, line width=0.5mm, rounded corners,inner sep=2pt}]
    \node[ve, label={\small{c}}] (v1) at (1, 0) {};
    \node[color=blue] at (1, -0.6) {\small{C}};
    \node[ve, label={}] (v2) at (0, -2) {};
    \node[ve, label={}] (v3) at (2, -2) {};
      \node at ([shift={(90:0.6)}]v2){\small{a\textcolor{blue}{A}}};
    \node at ([shift={(90:0.6)}]v3){\small{b\textcolor{blue}{B}}};
    \node [he1,  fit=(v1)] {};
    \node [he1,  fit=(v2)] {};
    \node [he1,  fit=(v3)] {};
 \node [he1, draw=cyan,inner sep=4pt, fit=(v2) (v3)]{};
   \node[color=cyan] at (1,-2){\small{D}};
 \draw [rounded corners,line width=0.5mm, draw=red, scale=1] (1, 1) --(-1,-2.5)--(3,-2.5)--cycle;
   \node[color=red] at (2.2,-0.5){\small{F}};
\end{tikzpicture}
    }
  \caption{$\mathcal{H}$}
  \label{fig:g1}
\end{subfigure}
\begin{subfigure}{.24\textwidth}
  \centering
  \resizebox{.9\textwidth}{!}{
  \begin{tikzpicture}[scale=1, ve/.style={fill=black,circle,scale=0.5},   he1/.style={draw = blue, line width=0.5mm, rounded corners=5pt,inner sep=3pt}]
   \draw [line width=0.5mm,fill=yellow!15] (1, 0)--(0,-2)--(2,-2)--cycle;
     \node[ve, color=blue, label={c}] (v3) at (1, 0) {};
    \node[ve, color=blue, label={a}] (v1) at (0, -2) {};
    \node[ve, color=blue, label={b}] (v2) at (2, -2) {};
\end{tikzpicture}
    }
  \caption{$\Delta(\mathcal{H})$}
  \label{fig:upper-closure}
\end{subfigure}
\begin{subfigure}{.24\textwidth}
  \centering
  \resizebox{.9\textwidth}{!}{
  \begin{tikzpicture}[scale=1, ve/.style={fill=black,circle,scale=0.5},   he1/.style={draw = blue, line width=0.5mm, rounded corners=5pt,inner sep=3pt}]
 \draw [line width=0.5mm,fill=yellow!15] (1, 0)--(0,-2)--(2,-2)--cycle;
      \node[ve, color=blue, label={c}] (v3) at (1, 0) {};
    \node[ve, color=blue, label={a}] (v1) at (0, -2) {};
    \node[ve, color=blue, label={b}] (v2) at (2, -2) {};
    \node[ve, label={}] (v12) at (1, -2) {};
    \node[ve,  label={}] (v13) at (0.5, -1) {};
      \node at ([shift={(170:0.3)}]v13.10){\tiny{ac}};
    \node[ve, label={}] (v23) at (1.5, -1) {};   
  \node at ([shift={(30:0.2)}]v23.10){\tiny{bc}};
 \draw[-] (v3)--(v12); 
 \draw (v1)--(v23);
 \draw (v2)--(v13);
\node[ve, color=red, label={}] (v123) at (1, -1.36) {}; 
  \node[ve, color=cyan, label={}] (v12) at (1, -2) {};
     \node[color=cyan] at ([shift={(40:0.15)}]v12.10){\tiny{ab}};
  \node[color=red] at ([shift={(18:0.2)}]v123.10){\tiny{abc}};
\end{tikzpicture}
    }
  \caption{$\mathcal{B}(\Delta(\mathcal{H}))$}
  \label{fig:2-3}
\end{subfigure}
\begin{subfigure}{.24\textwidth}
  \centering
  \resizebox{.9\textwidth}{!}{
  \begin{tikzpicture}[scale=1, ve/.style={fill=black,circle,scale=0.5},   he1/.style={draw = blue, line width=0.5mm, rounded corners=5pt,inner sep=3pt}
];
      \draw[line width=0.3mm,fill=yellow!15] (0, -2)--(2, -2)--(1, -1.38)--cycle;
    \node[ve, color=blue, label={c}] (v3) at (1, 0) {};
    \node[ve, color=blue, label={a}] (v1) at (0, -2) {};
    \node[ve, color=blue, label={b}] (v2) at (2, -2) {};
    \node[ve, color=cyan, label={}] (v12) at (1, -2) {};
     \node[color=cyan] at ([shift={(40:0.15)}]v12.10){\tiny{ab}};
 \draw[-] (v3)--(v12); 
\node[ve, color=red, label={}] (v123) at (1, -1.38) {}; 
  \node[color=red]  at ([shift={(18:0.2)}]v123.10){\tiny{abc}};
\end{tikzpicture}
    }
  \caption{$\mathcal{B}_{res}(\mathcal{H})$}
  \label{fig:res}
\end{subfigure}
\caption{An illustration of (a) a hypergraph $\H$ transformed into (b) its upper closure $\Delta(\H)$, (c) barycentric subdivision $\bary(\Delta(\H))$, and (d) restricted barycentric subdivision $\bary_{res}(\H)$. Lower case letters represent vertices, whereas upper case letters represent hyperedges. Concatenations (e.g.,~ab, abc) in (c) and (d) represent simplices in $\Delta(\H)$ such as edges and triangles.}
  \label{fig:rest_bary_ex}
\end{figure}

Given a hypergraph with maximum hyperedge size $N$, the barycentric subdivision has at least $2^N$ vertices, one for each set in the simplex corresponding to the maximum hyperedge. 
Even when $N$ is moderately sized, creating the barycentric subdivision can be computationally intractable. 
Instead, since we are only interested in vertices that correspond to hyperedges there is an alternative way to construct $\bary_{res}(\H)$ without first constructing $\bary(\Delta(\H))$ which requires the concept of a partial order or poset.
A \emph{poset}, $\mathcal{P} = \langle P, \leq \rangle$, is a set $P$ together with a binary relation $\leq$ that is reflexive ($p\leq p$ for all $p\in P$), transitive ($p\leq q$ and $q\leq r$ implies $p \leq r$), and antisymmetric ($p\leq q$ and $q\leq p$ implies $p=q$).
For $\H=(V, E)$ in $\Hyp$ without empty edges (it will become clear soon why we disallow empty edges) let $\ECP(\H)$ be its partial order of edge containment, i.e., $\ECP(\H) = (E, \subseteq)$.
As in the definition of the barycentric subdivision of a simplicial complex (Definition~\ref{def:rbs}), we construct $\bary_{res}^\ECP(\H)$ from $\ECP(\H)$ by adding a $k$-simplex for each $e_{0} \leq e_{1} \leq \cdots \leq e_{k}$ (not necessarily saturated) chain in $\ECP(\H)$. This is also known as the \emph{order complex} of $\ECP(\H)$.
We write $\bary_{res}^\ECP(\H)$ as a composition $\bary_{res}^\ECP(\H) := \Ord(\ECP(\H))$ where $\Ord$ is the order complex construction.
The vertices of $\bary_{res}^\ECP(\H)$ represent the singleton chains in $\ECP(\H)$, or equivalently the edges of $\H$.


\begin{proposition}
\label{prop:bary_res_alt}
Let $\H \in \Hyp$ be a hypergraph without multi-edges or empty edges. Then $\bary_{res}^\ECP(\H) = \bary_{res}(\H)$. 
\end{proposition}
\begin{proof}
The proof is straightforward.
By construction, and because we are assuming no multi-edges or empty edges, the vertex sets of $\bary_{res}^\ECP(\H)$ and $\bary_{res}(\H)$ are the same. 
Simplices in both represent a containment chain of sets. 
Therefore the structures must contain the same simplices.
\end{proof}

The ResBS is a case in which multi-edges and empty edges can cause complications.
If instead we start with an $\H$ in $\MHyp$ that has multi-edges then $\ECP(\H)$ will be a \emph{preorder} of edge containment\footnote{Consider edges $e\neq e'$ such that $\epsilon(e) = \epsilon(e')$. Then in $\ECP(\H)$ we will have $e \leq e'$ and $e' \leq e$ but $e \neq e'$.}. Additionally if $\H$ contains the empty set as a hyperedge then it will be a bottom element of $\ECP(\H)$ and cone off $\bary_{res}^\ECP$. In these cases it is not true that $\bary_{res}^\ECP(\H) = \bary_{res}(\H)$ since multi-edges are collapsed and empty edges are removed when creating $\Delta(\H)$ but not when creating $\ECP(\H)$. 
There are two possible remedies. First, one could consider $\bary_{res}^\ECP$ as the true definition of the ResBS, taking the order complex of the preorder of edge containment. Or we can first collapse $\H$ and remove empty edges and then proceed with either $\bary_{res}^\ECP$ or $\bary_{res}$ (due to Proposition \ref{prop:bary_res_alt}). 
Computationally it is preferable to use $\bary_{res}^\ECP$ and avoid constructing $\bary(\Delta(\H))$, so the question is whether or not to collapse edges in a multi-hypergraph or remove empty edges. 
Not collapsing may be desirable for real data as it allows for distinguishing between hypergraphs that have differing numbers of the same hyperedge. However, keeping the multi-edges requires working through the preorder category which complicates category theory arguments.
Empty edges also cause difficulty because they cone off $\bary_{res}^\ECP$ thus destroying any homological structure that may be present if the empty edge is ignored.
In the remainder of this section we will proceed assuming hypergraphs are in $\Hyp$ and do not contain empty edges. We leave it as an open question to consider the category theoretical implications of taking $\ECP(\H)$ for hypergraphs in $\MHyp$ without first collapsing edges and removing empty edges.

Since $\bary_{res}^\ECP(\H)$ and $\bary_{res}(\H)$ are equal for collapsed hypergraphs without empty edges, we use the notation $\bary_{res}(\H)$ but construct it via $\ECP(\H)$ to avoid computational blowup. 
Now we are ready to define the ResBS homology of a hypergraph.

\begin{definition}
The \emph{ResBS homology} of a hypergraph $\H \in \Hyp$ without empty edges is the homology of its ResBS, $\bary_{res}(\H)$,  i.e., $\hom{res}{\bullet}(\H) := \hom{}{\bullet}(\bary_{res}(\H))$.
\end{definition}

The ResBS homology was used in \cite{jenne2023stepping} to study hypergraphs constructed from cyber log data. 
The authors referred to the ResBS as the \emph{nesting complex}.
They studied a specific dataset and observed that homological features in dimension 1 often correspond to adversary behavior. The Betti numbers, $\betti{res}{k}$ (for $k=0,1$), observed during the benign time periods are smaller than those observed during the anomalous time periods.

Next we show that the number of connected components of $\bary_{res}(\H)$ is bounded above by the number of toplexes of $\H$. 
In our proof, we will need the notion of a maximal component of a poset.

\begin{definition}
Let $P=(X, \leq)$ be a poset. A set of vertices $C \subseteq X$ is a \emph{component} if, for every pair $x, y \in C$, there is a sequence $x=x_0, x_1, \ldots, x_k=y$ such that either $x_i \leq x_{i+1}$ or $x_i \geq x_{i+1}$ for all $0\leq i\leq k-1$. A component is \emph{maximal} if there is no other component $C'$ such that $C \subsetneq C' \subseteq X$.
\end{definition}

\begin{proposition}\label{prop:res_max}
Let $\H \in \Hyp$ without empty edges, then $\betti{res}{0}(\H) \leq |top(\H)|$ where $top(\H)$ is the set of toplexes of $\H$.
Equality holds if there are no hyperedges contained within a pairwise toplex intersection.
\end{proposition}
\begin{proof}
We first prove the bound and then show that it is sharp.
A hyperedge $e$ is a maximal element of $\ECP(\H)$ if and only if it is a toplex in $\H$. 
Each maximal component of $\ECP(\H)$ has at least one maximal element.
There is a bijection between components of $\bary_{res}(\H)$ and components of $\ECP(\H)$.
Therefore, the number of components of $\bary_{res}(\H)$, denoted as $\betti{res}{0}(\H)$, is at most the number of toplexes of $\H$, $|top(\H)|$.

To prove sharpness, we need to show that if there are no edges in the intersection of two toplexes, then every component has exactly one maximal element. 
We prove this by contradiction.
Assume that there is a component with at least two maximal elements, $T_1$ and $T_2$. 
Then by definition of a component there must be a path $T_1 = e_0, e_1, \ldots, e_k = T_2$ such that $e_i \leq e_{i+1}$ or $e_i \geq e_{i+1}$ for all $0\leq i \leq k-1$. 
Assume it is a minimum length path.
Start at $e_0$ and continue until the first $\leq$, and say this happens at $i$, so that $e_{i-1} \geq e_i \leq e_{i+1}$.
Since we are in a component, either $e_{i+1} \leq T_1$, or $e_{i+1} \leq T_2$, or $e_{i+1} \leq T$ for some other maximal element $T$ in the component. 
In the first case we can make a shorter path $e_0, e_{i+1}, \ldots, e_k$ which is a contradiction to the path being minimal.
In the last two cases we have that $e_i$ is less than both $T_1$ (via the descending chain $T_1 = e_0 \geq e_1 \geq \cdots e_{i-1} \geq e_i$) and either $T$ or $T_2$.
This is a contradiction to the assumption that there are no edges in the intersection of two toplexes.
Therefore, every component has exactly one maximal element giving us the desired equality, $\betti{res}{0}(\H) = |top(\H)|$.
\end{proof}

A stronger result can be found in Christopher Potvin's PhD thesis \cite[Theorem 3.2.4]{potvin2023hypergraphs}, namely that $\betti{res}{0}$ is equal to the number of \emph{fence components} of the hypergraph. A \emph{fence} in a hypergraph is a walk $e_{1}, e_{2}, \ldots, e_{k}$ such that for all $e_{j}$, either $e_{j-1} \subset e_{j} \supset e_{j+1}$ or $e_{j-1} \supset e_{j} \subset e_{j+1}$. Then a fence component of a hypergraph $\H$ is a maximal set of hyperedges such that for any pair of hyperedges $e$ and $f$, there is a fence between $e$ and $f$.

We close this section with a proof that the ResBS construction via the order complex of $\ECP(\H)$ is a functor from $\Hyp$ to $\Simp$.
As in the discussion around Proposition~\ref{prop:closure_functor}, since simplicial homology is a functor from $\Simp$ to $\Ab$, we have that $\hom{}{\bullet}\circ\bary_{res}^\ECP$ is a functor from $\Hyp$ to $\Ab$.


\begin{proposition}
\label{prop:bary_res_functor}
    $\bary_{res}^\ECP$ is a functor from $\Hyp$ to $\Simp$.
\end{proposition}

\begin{proof}
Since we can write $\bary_{res}^\ECP$ as the composition $\Ord \circ \ECP$, we can show that $\bary_{res}^\ECP$ is a functor by showing that both $\Ord : \Po \rightarrow \Simp$ and $\ECP : \Hyp \rightarrow \Po$ are functors.
First we must define the category $\Po$. The objects in $\Po$ are all finite posets and the morphisms are order preserving maps. In other words, if $F : P \rightarrow Q$ for $P, Q \in \Po$ then $p \leq_P p'$ implies $F(p) \leq_Q F(p')$. 

We have already defined the object map for $\ECP$, taking a hypergraph to its edge containment poset. Given two hypergraphs $\H_1=(V_1, E_1)$ and $\H_2=(V_2, E_2)$ and a map $f : V_1 \rightarrow V_2$ in $\Hyp$ we define $\ECP(f) := \po(f)$, the induced map on edges. 
We must show that $\po(f) : \ECP(\H_1) \rightarrow \ECP(\H_2)$ is a morphism in $\Po$, i.e., an order preserving map. 
Let $e_1 \leq_{\ECP(\H_1)} e_2$. Then $e_1 \subseteq e_2$ and $\po(f)(e_1) = \{f(v) : v \in e_1\} \subseteq \{f(v) : v \in e_2 \} = \po(f)(e_2)$.
Note that if $e_1 = \emptyset$ this logic still holds as $\po(f)(\emptyset) = \emptyset$.
Therefore, $\po(f)(e_1) \leq_{\ECP(\H_2)} \po(f)(e_2)$ as desired. Since $\po$ is a functor the identity and composition properties are inherited from $\Set$.

The object map for $\Ord$ takes posets to their order complex. Since the elements of a poset are the vertices of its order complex, the morphism map for $\Ord$ is the identity. To prove that $\Ord$ is a functor we must show that an order preserving poset map takes chains to chains. But this is immediate from the definition. Indeed, if $m : P \rightarrow Q$ is an order preserving map and $p_1 \leq_P p_2 \leq_P \cdots \leq_P p_k$, then $m(p_1) \leq_Q m(p_2) \leq_Q \cdots \leq_Q m(p_k)$. Since chains in the poset and simplices in the order complex are in bijection, this completes the proof.

Since both $\ECP$ and $\Ord$ are functors, their composition $\bary_{res}^\ECP$ is a functor from $\Hyp$ to $\Simp$.
\end{proof}

We thank Robby Green for discussions regarding this proof. Notice that for this proof we explicitly allowed hypergraphs to have empty edges as it does not cause complications. Since $\bary_{res}^\ECP$ and $\bary_{res}$ are equivalent on hypergraphs in $\Hyp$ without empty edges a corollary of Proposition \ref{prop:bary_res_functor} would be that $\bary_{res}$ is a functor from a full subcategory of $\Hyp$ where the objects do not contain empty edges to $\Simp$, and therefore $\hom{res}{\bullet}$ is a functor from that subcategory to $\Ab$.

We close this section by making a connection to a well-known relationship between finite topological spaces and finite posets. 
Specifically, there is a one-to-one correspondence between finite $T_0$-spaces (for every pair of points at least one of them has a neighborhood that does not contain the other) and finite posets. See \cite{stong1966finite}, or the more recent \cite[Section 4]{barmak2012strong}, for details on this correspondence. 
Moreover, McCord showed that finite posets (via their corresponding $T_0$-spaces) are weakly homotopy equivalent to their order complexes \cite{mccord_singular_1966}. Thus, they have isomorphic homology groups (singular homology for the finite posets, simplicial homology for the order complex).
The connection to this survey is via $\ECP(\H)$, the finite poset corresponding to $\H$, and its order complex, $Ord(\ECP(\H)) = \bary_{res}^\ECP(\H)$. 
These works on finite topologies then imply that the singular homology of the $T_0$-space associated to $\ECP(\H)$ is isomorphic to $\hom{res}{\bullet}(\H)$. In this survey we are focusing on simplicial, relative, and chain complex homology as those are more computable than singular homology. We invite interested readers to refer to the references in this paragraph for more details related to these connections.

\subsection{Relative barycentric subdivision homology}
\label{sec:RelBS}

The ResBS described in the prior subsection is concerned with only the portion of the barycentric subdivision that corresponds to hyperedges that are present in the hypergraph.
However, this construction loses any information about structure that might be contained in the \emph{missing} hyperedges. 
For example, consider a very simple hypergraph consisting only of the hyperedge $\{a,b,c\}$. 
The ResBS is only a single vertex representing that edge, and this is the same as the ResBS for any other hypergraph with a single hyperedge.
However, there is homological structure in the missing hyperedges, that is, they form an open triangle.
To address that issue, we introduce the relative barycentric subdivision (RelBS) homology of a hypergraph.
First, we need to define the missing subcomplex.
\begin{definition}
Let $\H=(V,E)$ be a hypergraph and $\bary(\Delta(\H))$ its barycentric subdivision. 
The \emph{missing subcomplex}, $\missing(\H)$, is the subcomplex of $\bary(\Delta(\H))$ induced by the vertices that represent sets that are not in $E$ (i.e., those sets that are not hyperedges of $\H$).
\end{definition}
Given this definition for $\missing(\H)$, we can define the RelBS homology using relative homology.
\begin{definition}
The \emph{RelBS homology} of a hypergraph $\H=(V, E)$ is the homology of $\bary(\Delta(\H))$ relative to $\missing(\H)$, i.e., $$\hom{rel}{\bullet}(\H) := \hom{}{\bullet}(\bary(\Delta(\H)), \missing(\H)).$$
\end{definition}

Unlike the ResBS homology, the RelBS homology does not have a version for $\MHyp$ without collapsing multi-edges and removing empty edges. 
It relies on simplicial relative homology and therefore $\Delta(\H)$ and $\missing(\H)$ must be simplicial complexes.

Intuitively, as discussed in Section \ref{sec:rel_hom}, we can consider $\hom{rel}{\bullet}$ as the reduced homology of the cell complex formed by collapsing all of the faces in $\missing(\H)$ down to a single point in $\bary(\Delta(\H))$. 
We still lose the structure contained within $\missing(\H)$ but we gain information about how the existing hyperedges are related via missing hyperedges.
For example, paths between existing subedges that transit through missing subedges manifest, in some cases, as $\hom{rel}{1}$ loops. 

\begin{example}[Figure~\ref{fig:rest_bary_ex} running example]
    To simplify notation, here and in later examples we use the notation $\langle X \rangle$ to be the vector space generated by the set $X$. 
    Consider the hypergraph $\H$ in Figure~\ref{fig:rest_bary_ex}.
    It is not difficult to see that $\betti{rel}{1} = 1$. 
    The missing subcomplex is $\missing(\H) = \{[ac], [bc]\}$ which essentially identifies those two vertices via the quotient operation.
    Since the dimension of $\missing(\H)$ is 0 we have $\Cgroup_2(\bary(\Delta(\H)), \missing(\H)) = \Cgroup_2(\bary(\Delta(\H)))$ and $\Cgroup_1(\bary(\Delta(\H)), \missing(\H)) = \Cgroup_1(\bary(\Delta(\H)))$. But $\Cgroup_0(\bary(\Delta(\H)), \missing(\H)) = \langle a, b, c, ab, ac, bc, abc \rangle / \langle ac, bc \rangle$. 
    To compute $\hom{1}{rel}(\H)$ we need $\ker \bdr_1$ and $\im \bdr_2$ in this relative chain complex. The image of $\bdr_2$ is generated by the boundaries of the six 2-simplices in $\bary(\Delta(\H))$. If we were just in $\bary(\Delta(\H))$ the kernel of $\bdr_1$ would be the same and there would be no homology in dimension 1. But since $\Cgroup_0$ is a quotient, in this case there is more in the kernel. In particular, $\bdr_1([ac, abc] + [abc, bc]) = (ac-abc)+(abc-bc) = ac-bc$ which is equivalent to zero since we quotient by $\langle ac, bc \rangle$.
    Therefore, the 1-chain $[ac, abc] + [abc, bc]$ in $\Cgroup_1(\bary(\Delta(\H)))$ is in the kernel of $\bdr_1$, but it is not in the image of $\bdr_2$. To complete this example one can show that this is the only additional basis element of $\ker\bdr_1$ (anything else would need to be a path between $ac$ and $bc$, which would be homologous by adding linear combinations of boundaries of elements of $\im\bdr_2$) and thus $\betti{rel}{1} = 1$.
\end{example}

\begin{example}
For another case in which structure is discovered via $\hom{rel}{\bullet}$ but not $\hom{res}{\bullet}$, consider the hypergraph, $\H$, with only one 3-edge, $E = \{\{a,b,c\}\}$, as introduced in the opening paragraph of this section. 
This hypergraph has the same $\bary(\Delta(\H))$ as in Figure~\ref{fig:rest_bary_ex}(c).
Informally (we will provide the details next), the quotient of $\bary(\Delta(\H))$ by $\missing(\H)$, which is induced by the vertices $\{a, b, c, ab, ac, bc\}$, now forms a hollow sphere as the entire boundary of $\bary(\Delta(\H))$ is collapsed to a single point.
However, if we were to add just a single hyperedge so that, for example, $E(\H') = \{\{a,b,c\}, \{a\}\}$, we would still have the same $\bary(\Delta(\H))$ and trivial $\hom{res}{\bullet}$ ($\bary_{res}(\H)$ is two vertices connected by an edge). But now $\missing(\H')$ is induced by the vertices $\{b, c, ab, ac, bc\}$ making $\hom{rel}{\bullet}$ also trivial because only part of the boundary of $\bary(\Delta(\H'))$ is identified.

The above provides the intuition, now we give the details. In both $\H$ and $\H'$ the $\Cgroup_2$ relative chain group is equal to the $\Cgroup_2$ chain group for the barycentric subdivision because the missing subcomplexes have no 2-simplices,
\[\Cgroup_2(\bary(\Delta(\H)), \missing(\H)) = \Cgroup_2(\bary(\Delta(\H)))/\Cgroup_2(\missing(\H)) = \Cgroup_2(\bary(\Delta(\H)))/0 = \Cgroup_2(\bary(\Delta(\H))).\]
$\Cgroup_2(\bary(\Delta(\H)))$ is induced by all six of the 2-simplices in $\bary(\Delta(\H))$ (and similarly for $\H'$).
The difference is in $\Cgroup_1$ and $\Cgroup_0$ where the missing subcomplex is different for $\H$ and $\H'$.
One can derive the following:
\begin{align*}
    \Cgroup_1(\bary(\Delta(\H)), \missing(\H)) &= \langle [ac, abc], [bc, abc], [ab, abc], [a, abc], [b, abc],[c,abc]\rangle\\
    \Cgroup_0(\bary(\Delta(\H)), \missing(\H)) &= \langle abc \rangle\\
    \Cgroup_1(\bary(\Delta(\H')), \missing(\H')) &= \langle [ac, abc], [bc, abc], [ab, abc], [a, abc], [b, abc],[c,abc], [a, ab], [a, ac]\rangle\\
    \Cgroup_0(\bary(\Delta(\H')), \missing(\H)') &= \langle abc, a \rangle
\end{align*}

To compute $\hom{rel}{2}$ for $\H$ and $\H'$ we need $\ker\bdr_2$ and $\im\bdr_3$. For both $\H$ and $\H'$, $\im\bdr_3=0$ since there are no 3-simplices. But the kernel of $\bdr_2$ is different. For a 2-chain, $c$, to be in the kernel we need $\bdr_2(c) = 0$. We will show the general form of the image of $\bdr_2$ and then determine what coefficients of $c$ make that image 0. Let $c=\alpha_1 [c, ac, abc] + \alpha_2 [c, bc, abc] + \beta_1 [a, ac, abc] + \beta_2 [a, ab, abc] + \gamma_1 [b, ab, abc] + \gamma_2 [b, bc, abc]$ be a generic 2-chain in $\Cgroup_2$, then
\begin{align*}
    \bdr_2(c) =& \alpha_1 ( [ac, abc]-[c, abc]+[c, ac])+\alpha_2 ( [bc, abc]-[c, abc]+[c, bc]) +\\
     &+\beta_1 ( [ac, abc]-[a, abc]+[a, ac])+\beta_2 ( [ab, abc]-[a, abc]+[a, ab]) +\\
     &+\gamma_1( [ab, abc]-[b, abc]+[b, ab])+\gamma_2( [bc, abc]-[b, abc]+[b, bc])\\
    =& [ac, abc](\alpha_1+\beta_1) + [bc, abc](\alpha_2+\gamma_2) + [ab, abc](\beta_2+\gamma_1)+\\
     &-[c, abc](\alpha_1+\alpha_2) -[a, abc](\beta_1+\beta_2) -[b, abc](\gamma_1+\gamma_2)+\\
     &+\alpha_1[c, ac] + \alpha_2[c, bc] + \beta_1[a, ac] + \beta_2[a, ab] + \gamma_1[b, ab] + \gamma_2[b, bc]
\end{align*}
In both $\Cgroup_1(\bary(\Delta(\H))$ and $\Cgroup_1(\bary(\Delta(\H'))$ some of these summands can be wiped away because of the quotient. In the case of $\H$ we get
\begin{align*}
\bdr_2(c) =& [ac, abc](\alpha_1+\beta_1) + [bc, abc](\alpha_2+\gamma_2) + [ab, abc](\beta_2+\gamma_1) + \\
  &-[c, abc](\alpha_1+\alpha_2) -[a, abc](\beta_1+\beta_2) -[b, abc](\gamma_1+\gamma_2).
\end{align*}
We see that letting $\alpha_i=\beta_i=\gamma_i=1$ for $i=1,2$ yields a $c\neq0$ for which $\bdr_2(c)=0$. Thus, $\ker\bdr_2$ is $\Z/2\Z$ for $\H$ and $\hom{rel}{2}(\H) = \Z/2\Z$.

In the case of $\H'$ we get 
\begin{align*}
\bdr_2(c) =& [ac, abc](\alpha_1+\beta_1) + [bc, abc](\alpha_2+\gamma_2) + [ab, abc](\beta_2+\gamma_1) + \\
  &-[c, abc](\alpha_1+\alpha_2) -[a, abc](\beta_1+\beta_2) -[b, abc](\gamma_1+\gamma_2) + \\
  &+\beta_1[a, ac] + \beta_2[a, ab].
\end{align*}
Here we must have $\beta_1=\beta_2=0$ which forces all other coefficients to be 0. Thus, $\ker\bdr_2=0$ for $\H'$ and $\hom{rel}{2}(\H') = 0$.

Similar computations can be done to show $\hom{rel}{1} = \hom{rel}{0} = 0$ for both $\H$ and $\H'$.
\end{example}

The recent thesis by Potvin also provides many theoretical results on the RelBS homology including a Mayer-Vietoris theorem \cite[Theorems 4.1.1 and 4.1.3]{potvin2023hypergraphs} and computational algorithms that bypass construction of $\Delta(\H)$ \cite[Section 5.2]{potvin2023hypergraphs}. 

Whether RelBS is a functor from $\Hyp$ to $\Ab$ has not yet been studied. However, even in the simple case of vertex identity maps and edge set inclusion, as would be typical in a hypergraph filtration, the barycentric subdivision can stay the same or grow while the missing subcomplex can shrink or grow depending on whether sub-edges or toplexes are added. Intuitively, this level of complexity already in the simple case of hypergraph inclusion leads us to believe functoriality is either not true or complicated, and therefore we leave it as an open question.


\subsection{Polar complex}
\label{sec:polar}
The polar complex was originally introduced in \cite[Definition 4]{itskov2020hyperplane} as a simplicial complex to study \emph{combinatorial codes} that arise from feedforward neural networks. The authors define a \emph{codeword}, $\sigma \subseteq \{1, \ldots, n\}=[n]$, and a \emph{combinatorial code} as a set of codewords, $\mathcal{C} \subseteq 2^{[n]}$. But this can be thought of as a hypergraph where $V=[n]$ and $E=\mathcal{C}$. Note that this definition allows the empty codeword (empty edge) but does not allow multi-edges so a combinatorial code is equivalent to a hypergraph in $\Hyp$. While the original paper uses the polar complex of a code to study stable hyperplane codes, we observe that the polar complex is yet another way to transform a hypergraph into a simplicial complex. It focuses not only on which vertices are in each hyperedge but also which vertices are \emph{not} in each hyperedge. We introduce it here using the language of hypergraphs.
\begin{definition}
\label{def:polar}
    Let $\H=(V, E)$ be a hypergraph. Define $\overline{V} := \{\overline{v} : v \in V\}$ to be a second copy of $V$ which can be distinguished from $V$ in notation. For every $e \in E$ define 
    \[ \Sigma(e) = \{v : v \in e\} \sqcup \{ \overline{v} : v \not\in e\} = e \sqcup \overline{V\setminus e}.\]
    Then the \emph{polar complex} of $\H$, $\pol(\H)$, is defined as the upper closure of $\{ \Sigma(e)\}$, that is, 
    \[ \pol(\H) = \Delta(\{ \Sigma(e) : e \in E \}).\]
\end{definition}
\noindent Since $\Gamma(\H)$ is an upper closure, each simplex must be unique; therefore, if $\H$ had multi-edges, the duplicate sets in $\Sigma(\H)$ would be collapsed when forming $\Gamma(\H)$. We may assume, then, that hypergraphs, like combinatorial codes, are in $\Hyp$ for the purposes of the polar complex.

The polar complex is a pure $(n-1)$-dimensional simplicial complex on the set $V\sqcup \overline{V}$, where $n=|V|$. (By \emph{pure}, we mean that each maximal simplex has the same dimension, namely, $n-1$.) Indeed, each $\Sigma(e)$ has size equal to $|V|$ since each vertex is in $\Sigma(e)$ either with or without a ``bar.'' 

\begin{definition}
The \emph{polar complex homology} of $\H$ is the homology of the polar complex of $\H$, i.e., 
\[\hom{pol}{\bullet}(\H) := \hom{}{\bullet}(\pol(\H)).\]
\end{definition}

It is not difficult to show some simple properties of the polar complex. For example, if $E(\H_1) \subseteq E(\H_2)$, then it must be the case that $\Gamma(\H_1) \subseteq \Gamma(\H_2)$. Additionally, $\Gamma(\H) \cong \Gamma(\overline{\H})$ where $\overline{\H}$ is the hypergraph formed by taking the complement of every edge. 
Slightly more advanced, in \cite[p. 348]{itskov2020hyperplane} the authors remark that the polar complex of the hypergraph containing all possible edges on $[n]$ is the boundary of the $n$-dimensional cross-polytope (a regular, convex, polytope that exists in $n$-dimensional space \cite{coxeter1973regular}). In particular, $\pol(2^{[2]}) = \pol(\{\emptyset, \{1\}, \{2\}, \{1, 2\}\})$ is the square and $\pol(2^{[3]})$ is the octahedron. It follows then that $\hom{pol}{p}(2^{[n]}) = 1$ if and only if $p=n-1$ or $p=0$, and 0 otherwise.

\begin{example}[Figure \ref{fig:rest_bary_ex} running example]
In Figure \ref{fig:polar_ex}, we show the polar complex for the hypergraph $\H$ in Figure \ref{fig:rest_bary_ex}(a). There are three missing faces from the octahedron: $a\overline{b}c$, $\overline{a}bc$, and $\overline{a}\overline{b}\overline{c}$. The polar complex homology of this example is $\hom{pol}{0}(\H) = \Z/2\Z$, $\hom{pol}{1}(\H) = (\Z/2\Z)^2$, $\hom{pol}{p}(\H) = 0$ for all $p > 1$.
Representatives of the two generators of $\hom{pol}{1}(\H)$ are highlighted in Figure \ref{fig:polar_ex}: 
\[G_1=\bar{a}c+\bar{a}\bar{c}+\bar{b}\bar{c}+
\bar{b}c \textrm{ (pink)}, \qquad G_2=bc+b\bar{c}+\bar{b}\bar{c}+\bar{b}c \textrm{ (teal)}.\]
Visually we see three 1-cycles, but they are homologous to $G_1$, $G_2$, and their sum:
\begin{description}
    \item[$\bar{a}\bar{b}+\bar{a}\bar{c}+\bar{b}\bar{c}$:] Homologous to $G_1$ via adding the boundary of the face $\bar{a}\bar{b}c$
    \item[$ac+a\bar{b}+\bar{b}c$:] Homologous to $G_2$ via adding the boundaries of the faces $a\bar{b}\bar{c}$, $ab\bar{c}$, and $abc$
    \item[$\bar{a}b+\bar{a}c+bc$:] Homologous to $G_1+G_2$ by adding the boundary of the face $\bar{a}b\bar{c}$
\end{description}

\begin{figure}[!ht]
\centering
\includegraphics[width=0.4\textwidth]{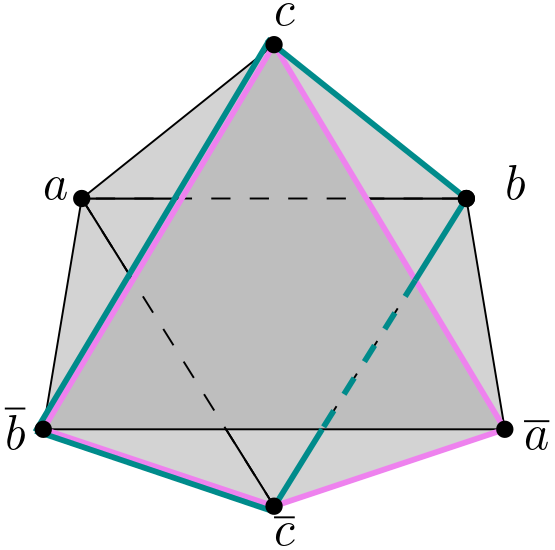}
\caption{The polar complex for the example hypergraph in Figure \ref{fig:rest_bary_ex}(a) with representatives of the two $\hom{pol}{1}(\H)$ generators shown in pink and teal.}
\label{fig:polar_ex}
\end{figure}
\end{example}

To answer the question of whether $\Gamma$ is a functor from $\Hyp$ to $\Simp$ we must first define the map on morphisms. For completeness we show both maps:
\begin{align*}
    \Gamma : \Hyp& \rightarrow \Simp\\
        (V, E)& \mapsto \Delta(\{\Sigma(e) : e \in E\}) \hspace{5em} \textrm{(object map)}\\
        f& \mapsto \hat{f}(v) = \left\{\begin{tabular}{ll}
                      $f(v)$ & $v \in V$\\
                      $\overline{f(\overline{v})}$ & $v \in \overline{V}$
                      \end{tabular}\right. \hspace{1em} \textrm{(morphism map)}
\end{align*}
In this definition, $\overline{\overline{v}} := v$. To make this concrete, if $f(a)=x$, then $\hat{f}(a) = x$ and $\hat{f}(\overline{a}) = \overline{x}$. 
However, this map is not a functor from $\Hyp$ to $\Simp$. A problem arises when $f$ is not one-to-one on $V$. Indeed consider the following simple case of hypergraphs and a morphism between them\footnote{We thank one of our anonymous reviewers for this example and for pointing out an error in an earlier version of Proposition \ref{prop:polar_functor}.}:
\[ \H_1 = (\{a,b\}, \{\{a\}, \{b\}\}), \quad \H_2 = (\{x\}, \{\{x\}\}), \quad f(a)=f(b)=x.\]
This is a valid morphism as it takes both edges of $\H_1$ to the edge $\{x\}$ in $\H_2$. However, the simplex $a\bar{b} \in \Gamma(\H_1)$ (corresponding to edge $\{a\}$) would be mapped to $x\bar{x}$, which cannot be in any polar complex. The ``barred'' vertices must be disjoint from the ``unbarred''. 
This shows concretely that this particular map, $\Gamma$, is not a functor from $\Hyp$ to $\Simp$. 
However, it can be shown to be a functor from a subcategory of $\Hyp$ that restricts to morphisms that are one-to-one. Since one-to-one morphisms may be of particular interest in the case of hypergraph filtrations (likely the most common way to apply persistence) we provide the proof here.

\begin{proposition}\label{prop:polar_functor}
    The map $\pol$ is a functor from a subcategory of $\Hyp$ where morphisms are required to be one-to-one to $\Simp$.
\end{proposition}

\begin{proof}
Let $\H_1=(V_1, E_1)$ and $\H_2=(V_2, E_2)$ be hypergraphs in $\Hyp$ and $f : V_1 \rightarrow V_2$ be a one-to-one morphism in $\Hyp$. To show that $\pol$ is a functor we must show that $\hat{f}$ is a simplicial map from $\Gamma(\H_1)$ to $\Gamma(\H_2)$.

Let $\sigma \in \pol(\H_1)$ be a facet. Then $\sigma$ corresponds to a hyperedge $e \in E_1$ and we can write $\sigma = e \sqcup \overline{V_1\setminus e}$.
Since $f$ is a morphism in $\Hyp$ the image of the vertices in $e$ must be a hyperedge in $\H_2$. Therefore, there is a facet in $\pol(\H_2)$ equal to $\{f(v) : v \in e\} \sqcup \overline{V_2 \setminus \{f(v) : v \in e\}}$. 
The vertices of the simplex $\sigma$ get mapped by $\hat{f}$ to $\{f(v) : v \in e\} \sqcup \{\overline{f(v)} : v \in V_1 \setminus e\}$. 
We need to show that $\{\overline{f(v)} : v \in V_1 \setminus e\} \subseteq  \overline{V_2 \setminus \{f(v) : v \in e\}}$.
But this is true since $f$ is one-to-one. Every vertex $v$ that is not in $e$ must be mapped into $V_2 \setminus \{f(v) : v \in e\}$.
Therefore the image of $\sigma$ under $\hat{f}$ is a subset of a facet in $\pol(\H_2)$ and therefore it is a simplex in $\pol(\H_2)$.

To show that a non-facet $\tau \in \pol(\H_1)$ maps to a simplex is then straightforward. $\tau$ must be a subset of a facet $\sigma$. Since we just showed that a facet maps to a simplex in $\pol(\H_2)$, all of its subsets must map to subsets of that simplex.
So $\hat{f}$ is a simplicial map. 

It is clear that $\Gamma(\textrm{id}_\H) = \textrm{id}_{\Gamma(\H)}$ by definition. And the fact that $\Gamma(f\circ g) = \Gamma(f)\circ\Gamma(g)$ is inherited from composition in $\Hyp$. 
\end{proof}

What we have not yet shown is that there is no morphism map that makes $\Gamma$ a functor from $\Hyp$ to $\Ab$. However, a bit more exploration into simple cases such as the one above, or including one more vertex, $c$, and edge, $\{c\}$, in $\H_1$ one can quickly see that defining a morphism map that works will be difficult or impossible.

\subsection{Embedded homology}
\label{sec:embedded}

The idea of embedded homology, introduced by Bressan et al.~\cite{bressan_embedded_2019}, differs from the closure, barycentric, and polar complex homologies discussed previously in that it is not a simplicial complex that is constructed from the hypergraph but rather a chain complex.
The construction starts with the chain complex from the upper closure and finds a particular sub-chain complex using only elements corresponding to the hyperedges such that the boundary maps are still valid. 
It is the chain complex homology (as described in Section~\ref{sec:chain-homology}) of this ``infimum chain complex'' that is called the \emph{embedded homology} of the hypergraph.
We give the formal details of this definition next.

Let $\Cgroup_p$ denote the $p$-dimensional chain group associated with $\Delta(\H)$. 
That is, each element $c \in \Cgroup_p$ is a linear combination of the $p$-simplices in $\Delta(\H)$ (with field coefficients).
Together with boundary maps $\bdr_p: \Cgroup_p \to \Cgroup_{p-1}$, we have the following chain complex, denoted as $\Cgroup_\bullet$:
\begin{align*}
\cdots \to \Cgroup_p \xrightarrow{\bdr_p} \Cgroup_{p-1} \to \cdots.
\end{align*}
For $p \geq 0$ let $\Dgroup_p$ denote the collection of all linear combinations of $(p+1)$-hyperedges in $Col(\H)$ (with field coefficients).
Then $\Dgroup_p$ is a subgroup of $\Cgroup_p$ for each $p$. 
The sequence of such subgroups is denoted  $\Dgroup_\bullet$. 
We use $Col(\H)$ in creating $\Dgroup_\bullet$ to remove multi-edges since $\Cgroup_\bullet$ is built on $\Delta(\H)$ which removes multi-edges in its creation. Without removing multi-edges we wouldn't be guaranteed to have $\Dgroup_p$ a subgroup of $\Cgroup_p$. As in the polar complex section we can simply assume $\H$ is in $\Hyp$ instead. If $\H$ contains an empty edge we ignore the edge and let $\Dgroup_{-1}$ be the zero vector space, as is typical for the purposes of homology, so that $\Dgroup_{-1} = \Cgroup_{-1}$.

Given a chain complex $\Cgroup_\bullet$ and a sequence of subgroups $\Dgroup_\bullet$, the \emph{infimum chain group} is defined as $\infgroup_p := \infgroup_p(\Dgroup_p, \Cgroup_p) = \Dgroup_p \cap \bdr^{-1}_{p}(\Dgroup_{p-1})$. 
We have the following commutative diagram:
\[ \begin{tikzcd}
   \Cgroup_p \arrow{r}{\bdr_p} & \Cgroup_{p-1} \\%
   \infgroup_p \arrow{r}{\delta_p} \arrow[swap]{u}{j_p} & \infgroup_{p-1} \arrow[swap]{u}{j_{p-1}} 
   \end{tikzcd}
\]
where $j_p$ are chain maps induced by inclusions $\infgroup_p \to \Dgroup_p \to \Cgroup_p$, and
$\delta_p = \bdr_p \mid_{\infgroup_p}$, i.e., $\delta_p$ is the restriction of $\bdr_p$ to $\infgroup_p$.
The chain map $j_p$ sends cycles to cycles, boundaries to boundaries, and thus induces a map on homology between $\Hgroup_p(\Delta(\H)):=\Hgroup_p(C_\bullet, \bdr_\bullet) = \myker \bdr_p / \myim \bdr_{p+1} $ and $\hom{emb}{p}(\H) := \Hgroup_p(\infgroup_\bullet, \delta_\bullet) = \myker \delta_p / \myim \delta_{p+1}$.  
As defined in~\cite[page 485]{bressan_embedded_2019}, $\hom{emb}{p}(\H)$ is the $p$-th \emph{embedded homology} of the hypergraph $\H$.

We refer the reader to \cite{bressan_embedded_2019} for many theoretical results on embedded homology including a Mayer-Vietoris sequence and a persistence formulation. In \cite[Proposition 3.7]{bressan_embedded_2019} Bressan, et al. essentially prove that embedded homology is a functor from $\Hyp$ to $\Ab$. The authors also provide some special cases of $\hom{emb}{\bullet}$ when $\H$ is acyclic (using the notion of acyclicity defined in \cite{berge1984hypergraphs}) and when $\H$ has one hyperedge that all other hyperedges are subsets of (i.e., if $\Delta(\H)$ is a single simplex and all of its subsimplices). 
Bressan et al. also introduced multiple measures for hypernetworks (from which hypergraph models are derived) that rely on embedded homology. These measures are aimed at analyzing the structure and features of hypergraphs derived from real-world data. They enable one to quantify such features as the interconnectedness of the vertices, the degree of differentiation or stratification within the hypernetwork, and the correlation between two functions defined on the vertices.

\begin{example}[Figure \ref{fig:rest_bary_ex} running example]
We return to our running example hypergraph from Figure \ref{fig:rest_bary_ex}(a) to show an example of embedded homology. In this example we have
\begin{center}
\begin{tabular}{lll}
    $\Cgroup_2 = \langle abc \rangle$        & $\Dgroup_2 = \langle abc \rangle$     & $\infgroup_2(\Dgroup_2, \Cgroup_2) = \langle abc \rangle \cap \bdr_2^{-1}(\langle ab \rangle)      = 0$ \\
    $\Cgroup_1 = \langle ab, ac, bc \rangle$ & $\Dgroup_1 = \langle ab \rangle$      & $\infgroup_1(\Dgroup_1, \Cgroup_1) = \langle  ab \rangle \cap \bdr_1^{-1}(\langle a, b, c \rangle) = \langle ab \rangle$ \\
    $\Cgroup_0 = \langle a, b, c \rangle$    & $\Dgroup_0 = \langle a, b, c \rangle$ & $\infgroup_0(\Dgroup_0, \Cgroup_0) = \langle a, b, c \rangle \cap \bdr_0^{-1}(\langle 0\rangle)                   = \langle a, b, c \rangle$
\end{tabular}
\end{center}
Given the infimum complex $\infgroup_2 \xrightarrow{\delta_2} \infgroup_1 \xrightarrow{\delta_1} \infgroup_0$, we see that $\hom{emb}{2}(\H) = \hom{emb}{1}(\H) = 0$ and $\hom{emb}{0}(\H) = (\Z/2\Z)^2$.
\end{example}

\begin{example}\label{ex:emb_intuition}
Another illustrative example to help build some intuition for the embedded homology is shown in Figure \ref{fig:emb_hom_uniform_cycles}.
Both hypergraphs have a cycle of edges (in the hypergraph sense), but it can be shown that $\betti{emb}{1}(\H_1) = 1$ whereas $\betti{emb}{1}(\H_2) = 0$.
To be explicit, for $\H_1$ it's easy to see that $\infgroup_2 = \langle 0 \rangle$ and $\infgroup_0 = \langle 0 \rangle$ (the latter because none of the singleton edges are present in $\H_1$). However, $\infgroup_1$ is more subtle. Since all the edges in the hypergraph cycle are size 2 their alternating sum is in the preimage of $\Dgroup_0=\langle 0 \rangle$.
\begin{align*}
    \Cgroup_1 &= \langle ab, bd, cd, ac \rangle \\
    \Dgroup_1 &= \langle ab, bd, cd, ac \rangle \\
    \Dgroup_0 &= \langle 0 \rangle \\
    \infgroup_1(\Dgroup_1, \Cgroup_1) &= \langle  ab, bd, cd, ac \rangle \cap \bdr_1^{-1}(\langle 0 \rangle) = \langle ab + bd - cd - ac \rangle
\end{align*}
Therefore, $\hom{emb}{1}(\H_1) = \myker \delta_1 / \myim \delta_{2} = \infgroup_1 / \langle 0 \rangle = \Z/2\Z$ and $\betti{emb}{1}(\H_1) = 1$. 

For $\H_2$ the story is different. 
Even though $\Dgroup_2 = \langle aeb \rangle$, we still have $\infgroup_2 = \langle 0 \rangle$ since the three edges of the $aeb$ triangle are not in $\Dgroup_1$. Then,
\begin{align*}
    \Cgroup_1 &= \langle ae, ab, eb, bd, cd, ac \rangle \\
    \Dgroup_1 &= \langle bd, cd, ac \rangle \\
    \Dgroup_0 &= \langle 0 \rangle \\
    \infgroup_1(\Dgroup_1, \Cgroup_1) &= \langle  bd, cd, ac \rangle \cap \bdr_1^{-1}(\langle 0 \rangle) = \langle 0 \rangle.
\end{align*}
In this case $\bdr_1^{-1}(\langle 0 \rangle)$ includes $ac+cd-bd-ab$ and $ac+cd-bd-ae+be$ but neither of these are in $\Dgroup_1$. With $\inf_2 = \inf_1 = \inf_0 = \langle 0 \rangle$ we have $\betti{emb}{1}(\H_2) = 0$.
\end{example}

Example \ref{ex:emb_intuition} boils down to the fact that the hypergraph cycle in $\H_1$ is made up of all edges of size 2 while in $\H_2$ one of the edges is of size 3. We can generalize this intuition by saying that embedded homology in dimension $p$ captures (among other things) uniform $p$-dimensional cycles of edges\footnote{That is, sets of hyperedges of size $p+1$ whose upper closure is a $p$-dimensional homological cycle. An example is the set of edges $abc, abd, acd, bcd$ that form a tetrahedron in the upper closure, even if none of the sub-edges are present in the hypergraph.} (as long as the enclosing $(p+2)$-edge is not present), even if not all (or even none!) of the sub-edges are present. This is in contrast to $\hom{\Delta}{\bullet}$ which would have $\betti{\Delta}{1}=1$ in both of these examples because the simplex $\{a, b\}$ would be present as a sub-edge of $\{a, b, e\}$ in the hypergraph on the right. 

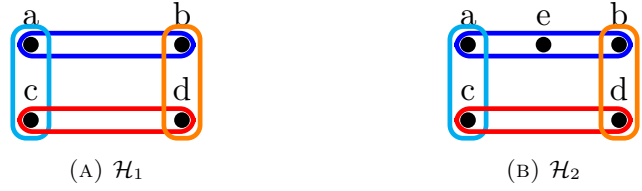
\begin{figure}
\centering
\begin{subfigure}{.45\textwidth}
  \centering
  \resizebox{.5\textwidth}{!}{
  \begin{tikzpicture}[scale=0.8, ve/.style={fill=black,circle,scale=0.5},  he1/.style={draw = blue, line width=0.5mm, rounded corners,inner sep=1pt}]
    %
    \node[ve, label={a}] (v1) at (0, 0) {};
    \node[ve, label={b}] (v3) at (2, 0) {};
    \node[ve, label={c}] (v4) at (0, -1) {};
    \node[ve, label={d}] (v6) at (2, -1) {};
    %
    \node [he1,  fit=(v1) (v3)] {};
    \node [he1,draw = cyan,inner sep=3pt, fit=(v1) (v4)]{};
 \node [he1,draw = red, fit=(v4) (v6)]{};
 \node [he1,draw = orange,inner sep=3pt, fit=(v3) (v6)]{};
\end{tikzpicture}
    }
  \caption{$\mathcal{H}_1$}
  \label{hyperg1}
\end{subfigure}
\begin{subfigure}{.45\textwidth}
  \centering
  \resizebox{.5\textwidth}{!}{
  \begin{tikzpicture}[scale=0.8, ve/.style={fill=black,circle,scale=0.5},  he1/.style={draw = blue, line width=0.5mm, rounded corners,inner sep=1pt}]
    %
    \node[ve, label={a}] (v1) at (0, 0) {};
    \node[ve, label={e}] (v2) at (1, 0) {};
    \node[ve, label={b}] (v3) at (2, 0) {};
    \node[ve, label={c}] (v4) at (0, -1) {};
    \node[ve, label={d}] (v6) at (2, -1) {};
    %
   \node [he1,  fit=(v1) (v2) (v3)] {};
    \node [he1,draw = cyan,inner sep=3pt, fit=(v1) (v4)]{};
 \node [he1,draw = red, fit=(v4) (v6)]{};
 \node [he1,draw = orange,inner sep=3pt, fit=(v3) (v6)]{};
\end{tikzpicture}
    }
  \caption{$\mathcal{H}_2$}
  \label{hyperg2}
\end{subfigure}
    \caption{Two hypergraphs to illustrate the fact that embedded homology captures ``uniform cycles.''}
    \label{fig:emb_hom_uniform_cycles}
\end{figure}

Research in embedded homology is quite active. In particular, previous papers have explored computational aspects of embedded homology \cite{liu2024computing}, developed relative embedded homology \cite{ren2022embedded} and a persistence theory of hypergraph embedded homology \cite{liu2021HPC}, and explored applications in the context of biology \cite{liu2021HPC}, face-to-face interactions \cite{liu2021HPC}, and politics \cite{volic2024hypergraphs}. 


\subsection{Hypergraph path homology}
\label{sec:path-homology}

The path homology for hypergraphs was introduced by Grigor'yan \etal~\cite{GrigoryanJimenezMuranov2019}. 
It builds upon the notion of path complexes and their homologies~\cite{GrigoryanLinMuranov2020}.
Here we reproduce the necessary definitions from both references to define hypergraph path homology.
A \emph{sequence} on a finite nonempty set $X$ is a string of elements of $X$ in which repetitions are allowed and order matters.  

\begin{definition}[{\cite[page 566]{GrigoryanLinMuranov2020}}]
For any nonnegative integer $p$, an \emph{elementary $p$-path} over $X$ is a sequence of elements in $X$, 
$[x_0, \dots, x_p]$, where $x_0, \dots, x_p$ do not have to be distinct. Brackets are used in the notation for a $p$-path because the sequence will be playing the role of a simplex in a chain complex.
\end{definition}

Let $\Lambda_p = \Lambda_p(X, \Kspace)$ denote the free vector space consisting of all formal linear combinations of elementary $p$-paths over $X$ with coefficients in a field $\Kspace$; for our simple examples, we work with $\Kspace = \Z/2\Z$. 
The elements of $\Lambda_p$ are called \emph{$p$-paths} on $X$.
Define $\Lambda_{-1} := \Kspace$ and $\Lambda_{-2} := \{0\}$.
For any $p$, we define a boundary operator $\bdr_p: \Lambda_{p} \to \Lambda_{p-1}$ that is a linear operator acting on elementary paths, 
\begin{align}
\bdr_p ([x_0, \dots, x_p]) := \sum_{i=0}^{p}(-1)^i[x_0, \dots, \hat{x}_i, \dots, x_p].
\label{eq:bdr}
\end{align}
Again, $\hat{x}_i$ means the omission of $x_i$.
It can be shown that $\bdr_p \circ \bdr_{p+1} = 0$, and $\Lambda_p$ and $\bdr_p$ give rise to a standard chain complex, and an augmented chain complex, respectively:   
\[
\cdots  \to \Lambda_p \to \Lambda_{p-1} \to \cdots \to \Lambda_0 \to 0, 
\]
\[
\cdots  \to \Lambda_p \to \Lambda_{p-1} \to \cdots \to \Lambda_0 \to \Kspace \to 0.  
\]

Regular paths are special types of elementary paths where a certain kind of repetition is not allowed. 
\begin{definition}[{\cite[page 567]{GrigoryanLinMuranov2020}}]
An elementary $p$-path $[x_0, \dots, x_p]$ is \emph{regular} if adjacent elements are distinct, that is, $x_i \neq x_{i+1}$ for $0 \leq i \leq p-1$; otherwise it is \emph{irregular (or nonregular)}. 
\end{definition} 
Let $\Rcal_p:=\Rcal_p(X, \Kspace)$ denote the span of the regular $p$-paths, \ie, the set of all finite linear combinations of regular $p$-paths. 
Let $\Ical_p:=\Ical_p(X, \Kspace)$ denote the span of the irregular $p$-paths. 
It has been shown that $\bdr_\bullet$ defined in \eqref{eq:bdr} is not invariant on the family $\{\Rcal_p\}$, but is invariant on $\{\Ical_p\}$.
It has also been shown that $\Rcal_p \cong \Lambda_p / \Ical_p$ via a natural linear isomorphism; we can define (with an abuse of notation) $\bdr_p: \Rcal_p \to \Rcal_{p-1}$ as the pullback of the boundary map $\bdr_p: \Lambda_p \to \Lambda_{p-1}$ via this isomorphism. 
In simpler terms, the formula for $\bdr_p$ defined on $\Rcal_p$ is the same as above with any irregular terms that appear on the right-hand side treated as zero.

Whereas $\Lambda_p$ (resp. $\Rcal_p$, $\Ical_p$) are generated by all (resp. regular, irregular) elementary $p$-paths, path homology is concerned with special sets of elementary $p$-paths, those belonging to a path complex, which we define next.
The notion of a path complex over a finite nonempty set $X$ generalizes the notions of a simplicial complex and a digraph. 
\begin{definition}
[{\cite[Definition 3.1]{GrigoryanLinMuranov2020}}]
A \emph{path complex} over a set $X$ is a nonempty collection $P:=P(X)$ of elementary paths on $X$ such that if $[x_0, \dots, x_p] \in P$, then $[x_1, \dots, x_p] \in P$ and $[x_0, \dots, x_{p-1}] \in P$. 
\label{def:path-complex} 
\end{definition}
When a path complex $P$ is fixed, all the paths from $P$ are called \emph{allowed elementary paths}. 
Definition~\ref{def:path-complex} means that if we remove the first or the last element of an allowed $p$-path, then the resulting $(p-1)$-path is also allowed.
Denote by $P_p$ the set of all $p$-paths from $P$.
The set $P_{-1}$ contains a single empty path.  
Following  \cite{GrigoryanLinMuranov2020}, the elements of $P_0$ are called the \emph{vertices} of $P$, and $P_0 \subseteq X$. 
For simplicity, we will remove from the set $X$ all nonvertices so that $P_0 = X$. 
A path complex contains elementary paths of varying lengths, $P = \{P_p\}_{p=0}^{\infty}$.  
 
Given a path complex $P$ over a nonempty finite set $X$, we consider the linear space $\Acal_p: = \Acal_p(P, \K)$ that is spanned by all the elementary $p$-paths in $P_p$ (for any integer $p \geq -1$) with coefficients from $\K$. 
The elements of $\Acal_p$ are called \emph{allowed $p$-paths}. 
$\Acal_p$ is a subspace of $\Lambda_p$ by construction. 
We will restrict the operator $\bdr_p$ defined on spaces $\Lambda_p$ to the subspaces $\Acal_p$. 
In general,  $\bdr_p (\Acal_p)$ does not have to be a subspace of $\Acal_{p-1}$.
Therefore, we define a ``well-behaved'' subspace\footnote{Note the similarity of this definition to the infimum complex of embedded homology. Indeed the construction of the chain complex from a sequence of groups is similar though the groups are quite different.}, 
$\Omega_p = \Omega_p(P) = \Acal_p \cap \bdr_p^{-1}(\Acal_{p-1})$.
The elements of $\Omega_p$ are called \emph{$\bdr$-invariant (boundary invariant) $p$-paths}.
We thus obtain a standard chain complex of $\bdr$-invariant paths, and an augmented one, respectively:  
\[
\cdots  \to \Omega_{p} \to \Omega_{p-1} \to \cdots \to \Omega_{0} \to 0,   
\]
\[
\cdots  \to \Omega_{p} \to \Omega_{p-1} \to \cdots \to \Omega_{0} \to \Kspace \to 0.  
\]
The homology groups of the above chain complexes are referred to as the \emph{path homology} and  \emph{reduced path homology} of the path complex $P$, denoted by $\Hgroup_p(P)$ ($p \geq 0$) and  $\tilde{\Hgroup}_p(P)$ ($p \geq -1$) respectively, for $p \geq 0$~\cite{GrigoryanLinMuranov2020}.

\begin{definition}[{\cite[Definition 3.4]{GrigoryanLinMuranov2020}}]
A path complex $P$ is \emph{regular} if all the paths $[x_0, \dots, x_p]$ in $P$ are regular. 
\end{definition}

In this case $\Acal_p$ is a subspace of $\Rcal_p$, so we can replace the irregular boundary operator $\bdr_p$ (defined on $\Lambda_p$) in $\Omega_p = \Omega_p(P) = \Acal_p \cap \bdr_p^{-1}(\Acal_{p-1})$ by the regular boundary operator $\bdr_p$ (defined on $\Rcal_p$).  
The resulting space $\Omega_p^{reg}$ is referred to as a \emph{regular space of $\bdr$-invariant paths}. 
This gives rise to another standard chain complex, together with an augmented one:
\[
\cdots  \to \Omega_p^{reg} \to \Omega_{p-1}^{reg} \to \cdots \to \Omega_0^{reg} \to 0,   
\] 
\[
\cdots  \to \Omega_p^{reg} \to \Omega_{p-1}^{reg} \to \cdots \to \Omega_0^{reg} \to \Kspace \to 0.  
\]
Homology groups of the above chain complexes are called the \emph{regular path homology} and \emph{reduced regular path homology} of $P$, denoted as $\Hgroup^{reg}_p(P)$ ($p \geq 0$) and $\tilde{\Hgroup}^{reg}_p(P)$ ($p \geq -1$), respectively.

To define the path homology of a hypergraph $\H=(V, E)$, we define a path complex from $\H$ and then study its homology. 
\begin{definition}[{\cite[Definition 5.5]{GrigoryanJimenezMuranov2019}}]
For a hypergraph $\H=(V, E)$, define a path complex of density $q \geq 1$ on the set $V$ of vertices, denoted as $P^q(\H)$, in the following way: 
a path of length $p \geq 0$ and density $q$ is defined as a sequence of $p+1$ vertices such that any $q$ consecutive vertices lie in some hyperedge of $\H$.
\end{definition}

\begin{definition}[{\cite[Definition 5.12]{GrigoryanJimenezMuranov2019}}]
\label{def:hypergraph-path-homology}
For a hypergraph $\H$ and any $q \geq 1$, we have a standard chain complex 
\[
\cdots  \to \Omega_{p}(P^q(\H)) \to \Omega_{p-1}(P^q(\H)) \to \cdots \to \Omega_{0}(P^q(\H)) \to 0.     
\]
The homology groups of the above chain complex are called the \emph{path homology of a hypergraph}, denoted $\Hgroup_p^{path}(\H, q)$. 
The \emph{regular path homology of a hypergraph}, denoted $\Hgroup^{reg}_p(\H, q)$, together with their reduced versions, are defined similarly. 
\end{definition}

A path complex of a hypergraph, $P^q(\H)$, is unaffected by the presence of both multi-edges and empty edges. If a sequence of $p+1$ vertices has all $q$ consecutive vertices in some hyperedge of $\H$ the fact that they may be in multiple copies of the same hyperedge is immaterial. Additionally an empty edge cannot contribute anything to the path complex because no vertices can be in an empty edge together.
Therefore, path homology is defined for any $\H \in \MHyp$, though multi and empty edges are ignored.
We formalize this observation with a proposition that the maximal edges determine the path homology of a hypergraph. This implies that path homology, like closure homology, is constant on hyperblocks. 
\begin{proposition}[{\cite[Proposition 5.6]{GrigoryanJimenezMuranov2019}}]\label{prop:path_constant_hyperblock}
    The path homologies of a hypergraph, $\Hgroup_p^{path}(\H, q)$ and $\Hgroup^{reg}_p(\H, q)$, depend only on the set of its maximal edges.
\end{proposition}

We continue by computing the path homology of our running example.

\begin{example}[Figure \ref{fig:rest_bary_ex} running example]
We return again to the hypergraph in Figure~\ref{fig:rest_bary_ex}(a) and show its regular path homology for density $q=2$. First we observe that every pair of vertices is in an edge together because edge $F$ contains all vertices. Therefore, the hypergraph regular path complex $P^2_{reg}(\H)$ consists of all regular paths on $\{a, b, c\}$:
\begin{align*}
    P^2_{reg}(\H) : \quad p &= 0 \quad \{a, b, c\}\\
                    p &= 1 \quad \{ab, ac, ba, bc, ca, cb\}\\
                    p &= 2 \quad \{aba, abc, bab, bac, aca, acb,\\
                    \phantom{p}&\phantom{= 2 \quad \{,,}  bca, bcb, cab, cac, cba, cbc\}
\end{align*}  
Then it's not difficult to show that $\Omega_\bullet$ is generated by all of $P^2_{reg}(\H)$ at each length $p$. For $p=0$ and $p=1$ it's immediate, and for $p=2$ it's easy to show that the regular boundary map applied to an arbitrary linear combination of all elementary 2-paths (those in $P^2_{reg}(\H)$ for p=2) is a linear combination of elementary 1-paths (those in $P^2_{reg}(\H)$ for p=1) so the preimage of the group generated by all the elementary 1-paths is generated by all of the elementary 2-paths.
From here it's a linear algebra exercise to show that $H_0^{reg}(\H, 2) = \Z/2\Z$ and $H_1^{reg}(\H, 2) = 0$.
\end{example}

We close this subsection by citing some results on functoriality of path homology from \cite{GrigoryanJimenezMuranov2019}. In particular Grigor'yan, et al. show that $P^q$ is a functor from a category of hypergraphs nearly equivalent to $\MHyp$\footnote{They use the category from \cite{dorfler1980category} which disallows empty edges.} to a category of path complexes (Proposition 5.7). They also show that $\Omega_\bullet$ and $\Omega^{reg}_\bullet$ are functors from a category of path complexes to the category of chain complexes (Corollary 2.5), and that $\Omega^{reg}_\bullet$ is additionally a functor from a category of path complexes with ``weak morphisms'' (which they define) to the category of chain complexes (Corollary 2.10). Therefore, the transformation from a hypergraph to the regular and irregular chain complexes of a hypergraph path complex is a functor from the category of hypergraphs to the category of chain complexes (Corollary 5.9).

\subsection{Weighted nerve and persistence} 
\label{sec:weighted}

Thus far in this survey, we have presented hypergraph homology theories that consist of one computation of homology on a single topological object: a simplicial complex, chain complex, or relative chain complex. But a commonly used tool in TDA is \emph{persistent homology} (PH), which considers a filtration or a nested sequence of objects, e.g., simplicial complexes, with inclusion maps from one to the next. If one computes homology of each of the objects in the sequence, the inclusion maps induce maps on the homology groups (since homology is a functor). This allows one to track the appearance and disappearance of topological features along indices of the filtration. We refer the reader to some standard introductions to persistent homology for details \cite{Carlsson2009TopologyAD,dey2022computational, edelsbrunner2010computational, EdelsbrunnerLetscherZomorodian2002,ghrist2014elementary, zomorodian2005computing}. 

The last hypergraph homology theory that we present in detail in this survey is the persistent homology of a filtration of a weighted simplicial complex. This weighted simplicial complex has the property that the hypergraph can be fully (and uniquely up to isomorphism) reconstructed from it.
We begin by defining the weighted nerve of a hypergraph and then prove the reconstruction property.
\begin{definition}\label{def:wtd_nrv}
The \emph{weighted nerve complex}, or simply \emph{weighted nerve}, of a hypergraph $\H$, denoted $\Nrv_w(\H)$, is the nerve of $\H$ (recall from Definition~\ref{def:nerve}) weighted by $w(\sigma) = |\cap_{e \in \sigma} e|$.
\end{definition}

If multi-edges are present in $\H$ then they will be represented by distinct vertices in $\Nrv_w(\H)$, therefore $\Nrv_w(\H)$ applies to $\MHyp$ without collapsing multi-edges.
Empty edges are also allowed and will be represented as isolated vertices with weight 0 in $\Nrv_w(\H)$.
However, if isolated vertices are present they are not captured in $\Nrv_w(\H)$. 
An example weighted nerve for the hypergraph in Figure \ref{fig:wtd_nerve_ex}(a) is shown in Figure \ref{fig:wtd_nerve_ex}(b) and will be discussed in detail in Example \ref{ex:wtd_nerve_ex}. This is very close to the running example in Figure \ref{fig:rest_bary_ex}(a) now with multi-edges ($C_1$ and $C_2$) and duplicate vertices ($c_1$ and $c_2$).

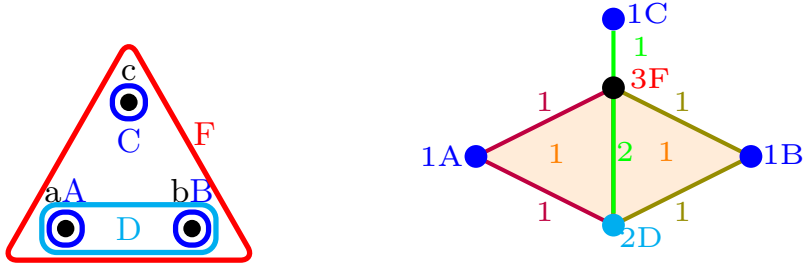
\begin{figure}[h!]
  \centering
  \begin{subfigure}{.35\textwidth}
  \centering
  \resizebox{0.9\textwidth}{!}{
  \begin{tikzpicture}[scale=0.6, ve/.style={fill=black,circle,scale=0.5},  he1/.style={draw = blue, line width=0.5mm, rounded corners,inner sep=2pt}]
    \node[ve] (v1) at (0.5, 0) {};
    \node[ve] (v4) at (1.5, 0) {};
    \node[color=black] at (0.5, 0.7) {\small{$c_1$}};
    \node[color=black] at (1.5, 0.7) {\small{$c_2$}};
    \node[color=orange] at (0.5, -0.8) {\small{$C_1$}};
    \node[color=purple] at (1.5, -0.8) {\small{$C_2$}};
    \node[ve, label={}] (v2) at (0, -2) {};
    \node[ve, label={}] (v3) at (2, -2) {};
      \node at ([shift={(90:0.6)}]v2){\small{a\textcolor{blue}{A}}};
    \node at ([shift={(90:0.6)}]v3){\small{b\textcolor{blue}{B}}};
    \node [he1, draw=orange, fit=(v1) (v4)] {};
    \node [he1, draw=purple, inner sep=5pt, fit=(v1) (v4)] {};
    \node [he1,  fit=(v2)] {};
    \node [he1,  fit=(v3)] {};
 \node [he1, draw=cyan,inner sep=4pt, fit=(v2) (v3)]{};
   \node[color=cyan] at (1,-2){\small{D}};
 \draw [rounded corners,line width=0.5mm, draw=red, scale=1] (2,1)--(0,1)--(-1,-2.6)--(3,-2.6)--cycle;
   \node[color=red] at (2.7,-0.5){\small{F}};
\end{tikzpicture}
    }
    \caption{}
  \label{fig:6g1}
\end{subfigure}
 \begin{subfigure}{.6\textwidth}
  \centering  \resizebox{0.7\textwidth}{!}{
  \begin{tikzpicture}[ve/.style={fill=black,circle,scale=0.5},   he1/.style={draw = blue, line width=0.5mm, rounded corners=5pt,inner sep=3pt}
];
      \draw[line width=0.3mm, color=olive,fill=yellow!15] (0, -1)--(1, -1.5)--(1, -0.5)--cycle;
      \node[color=olive]  at (0.5,-0.6) {\tiny{1}};
      \node[color=olive]  at (0.5,-1.4) {\tiny{1}};
      \draw[line width=0.3mm,color=olive,fill=yellow!15] (2, -1)--(1, -1.5)--(1, -0.5)--cycle; 
      \draw[line width=0.3mm,color=olive,fill=yellow!15] (0.5, 0)--(1.5, 0)--(1, -0.5)--cycle; 
        \node[color=olive]  at (1.5,-0.6) {\tiny{1}};
       \node[color=olive]  at (1.5,-1.4) {\tiny{1}};
    \node[ve, color=orange] (v3) at (0.5, 0) {};
    \node[ve, color=purple] (v4) at (1.5, 0) {};
    \node[ve, color=blue] (v1) at (0, -1) {};
   \node[ve, color=blue] (v2) at (2, -1) {};
   \node[color=blue]  at (-0.25,-1){\tiny{1A}}; 
  \node[color=blue]  at ([shift={(1:0.15)}]v2.10){\tiny{1B}}; 
    \node[color=orange] at (0.5,0.2) {\tiny{2$C_1$}};
    \node[color=purple] at (1.5,0.2) {\tiny{2$C_2$}};
    \node[ve, color=cyan] (v12) at (1, -1.5) {};
     \node[color=cyan, below] at (1.1, -1.5){\tiny{2D}}; 
   \node[ve, color=red] (v123) at (1, -0.5) {}; 
 \draw[line width=0.3mm, color=olive] (v12)--(v123); 
\node[color=red]  at ([shift={(18:0.19)}]v123.10){\tiny{4F}}; 
\node[color=violet]  at ([shift={(1:0.5)}]v1.10){\tiny{1}}; 
\node[color=violet]  at ([shift={(1:1.3)}]v1.10){\tiny{1}}; 
\node[color=olive]  at ([shift={(1:1)}]v1.10){\tiny{2}}; 
\node[color=olive]  at (1.4,-0.25) {\tiny{2}};
\node[color=olive]  at (0.6,-0.25) {\tiny{2}};
\node[color=olive]  at (1,0.1) {\tiny{2}};
\node[color=violet]  at (1,-0.25) {\tiny{2}};
\end{tikzpicture}}
  \caption{}
  \label{fig:6nerve}
\end{subfigure}
 \caption{The weighted nerve (b) of the hypergraph in (a). In (b) vertices are labeled by their weight followed by the name of the hyperedge they represent.}
    \label{fig:wtd_nerve_ex}
\end{figure}

Before we define a filtration of the weighted nerve to apply persistence, we prove that if two hypergraphs (without isolated vertices) have the same weighted nerve, they must be isomorphic. 

\begin{proposition}\label{prop:unique_wtd_nerve}
    If $\H_1$ and $\H_2$ are hypergraphs with no isolated vertices such that $\Nrv_w(\H_1) \cong \Nrv_w(\H_2)$, then $\H_1 \cong \H_2$.
\end{proposition}


The proof uses a M\"{o}bius inversion argument and is constructive, which implies that we can uniquely reconstruct a hypergraph from a weighted nerve up to relabeling of the vertices.
The M\"obius inversion theorem states that, given a weighted poset such that $w(x) = \sum_{y\leq x} f(y)$, the function $f$ is uniquely defined in terms of $w$ and the M\"obius function on the poset. See \cite[Sec. 3.7]{stanleyEC1} for more details. In our setting, we define a poset on the simplices of $\Nrv(\H)$ for a hypergraph $\H=(V, E)$ such that $\tau \leq \sigma$ if and only if $\tau \supseteq \sigma$. Note that this means $\cap_{e \in \tau} e \subseteq \cap_{e \in \sigma} e$. We first prove a helpful lemma and then return to the proof of Proposition~\ref{prop:unique_wtd_nerve}.

\begin{lemma}\label{lem:unique_f}
    Let $\H=(V, E)$ without isolated vertices and $\Nrv_w(\H)$ as defined above, then there is a unique function $f : \Nrv(\H) \rightarrow \R$ such that $w(\sigma) = \sum_{\tau \leq \sigma} f(\tau)$.
\end{lemma}
\begin{proof}
    By the M\"obius inversion theorem, if such an $f$ exists, then it must be unique. Therefore, it is enough to construct an $f$ that satisfies the desired equality. For every $v \in V$, let $\sigma(v) := \{e \in E : v \in e\}$, that is, the set of all edges that $v$ is contained in. Note that $\sigma(v) \in \Nrv(\H)$ because the intersection of the edges in this set must contain at least $v$. And $\sigma(v)$ will never be empty because $\H$ has no isolated vertices. 
    
    Define a function $f$ as follows:
    \[ f(\sigma) := \sum_{v \in \cap_{e \in \sigma} e} \mathbbm{1}_{\{ \sigma(v) = \sigma \}}. \]
    In words, $f(\sigma)$ is the number of vertices $v \in \cap_{e \in \sigma} e$ such that $v$ is in exactly the set of edges defined by $\sigma$ ($v$ is in no additional edges).
    To simplify this definition, we remark that if $v$ is not in $\cap_{e \in \sigma} e$ then it is definitely not the case that $\sigma(v)=\sigma$ (there is at least one edge in $\sigma$ that $v$ is not in). Therefore, we can write
    \[ f(\sigma) := \sum_{v \in V} \mathbbm{1}_{\{ \sigma(v) = \sigma \}}. \]
    
    Consider the sum of $f$ over all simplices below and including $\sigma$:
    \begin{align*}
        \sum_{\tau \leq \sigma} f(\sigma) &= \sum_{\tau \leq \sigma} \sum_{v \in V} \mathbbm{1}_{\{\sigma(v) = \tau\}}\\
        &= \sum_{\tau \leq \sigma} \left| \left\{ v \in V: \sigma(v) = \tau \right\} \right|\\
        &= \left| \bigcup_{\tau \leq \sigma} \left\{ v \in V : \sigma(v) = \tau \right\} \right|.
    \end{align*}
    The final step in the chain of equalities above is true because the sets in the sum are disjoint for all $\tau$. Indeed, each $v$ has a unique $\sigma(v)$ and so is contained in only one such set. To show that $w(\sigma) := |\cap_{e \in \sigma} e| = \sum_{\tau \leq \sigma} f(\sigma)$ we need only show that 
    \[\bigcap_{e \in \sigma} e = \bigcup_{\tau \leq \sigma} \left\{ v \in V : \sigma(v) = \tau \right\}.\]

    For the forward containment, let $v \in \cap_{e \in \sigma} e$. Then $\sigma(v) \leq \sigma$ so there exists some $\tau \leq \sigma$ such that $v$ is an element of the set on the right.  On the other hand, let $v \in \cup_{\tau \leq \sigma} \left\{ v \in V : \sigma(v) = \tau \right\}$. Then there is $\tau \leq \sigma$ such that $\sigma(v) = \tau$, that is, $v \in \cap_{e \in \tau} e$. Since $\tau \leq \sigma$, $\sigma$ has fewer edges and so it must be that $v \in \cap_{e \in \sigma} e$.
\end{proof}

In the preceding proof, we defined $f(\sigma)$ to be the \emph{number} of vertices in $V$ that are in exactly the set of edges defined by $\sigma$. Of course, we can also define the \emph{set} $F(\sigma) := \{v \in V : \sigma(v) = \sigma\}$ for each $\sigma \subseteq E$ excluding $\emptyset$ (any $\tau \not\in \Nrv(\H)$ has empty intersection and thus $f(\tau)=0$ and $F(\tau)=\emptyset$). 
Since each vertex, $v$, is only in $F(\sigma(v))$, all of the $F(\sigma)$ sets are disjoint.
This is true even if there are multi-edges or ``duplicate'' vertices (sets of vertices in the exact same set of edges). 

Thus far we have constructed $\sigma(v)$, $f$, and then $F$ starting from a hypergraph, $\H$. 
But notice that if we just had a map $\sigma \mapsto F(\sigma)$ such that all of the $F(\sigma)$ are disjoint then this uniquely determines the incidence structure of a hypergraph that would give rise to the same $F$. Indeed, the $\sigma$ for which $v \in F(\sigma)$ is $\sigma(v)$, which tells us exactly the set of hyperedges that $v$ is in.
We will now use Lemma~\ref{lem:unique_f} and this discussion to prove Proposition \ref{prop:unique_wtd_nerve}.

\begin{proof}[Proof of Proposition \ref{prop:unique_wtd_nerve}]
    For simplicity, since the two nerves are isomorphic, WLOG we relabel the vertices of $\Nrv_w(\H_2)$ using the isomorphism so that $\Nrv_w(\H_1) = \Nrv_w(\H_2) =:\Nrv_w$. This implies a relabeling of the edges of $\H_2$.
    
    Using the construction in the proof of Lemma \ref{lem:unique_f} we can define an $f_1$ from $\H_1$ and an $f_2$ from $\H_2$ so that for all $\sigma \in \Nrv_w$,
    \[ \sum_{\tau \leq \sigma} f_1(\tau) = w(\sigma) = \sum_{\tau \leq \sigma} f_2(\tau).\]
    By Lemma \ref{lem:unique_f} it must be that $f_1 = f_2 =: f$. 
    Since these $f$ functions are equal, and because of how they are defined, it must be the case that for each intersection of hyperedges (i.e., each simplex, $\sigma$) the number of vertices that are in exactly that set of hyperedges ($f(\sigma)$) is the same for both $\H_1$ and $\H_2$. 
    What is left to show is that given such an $f$ on $\Nrv_w$, there is a unique hypergraph (up to relabeling of vertices) that can give rise to it.
    By the discussion above, between the proof of Lemma \ref{lem:unique_f} and this proof, we merely need to construct a map $F : \Nrv_w \rightarrow \po(V)$ such that $|F(\sigma)| = f(\sigma)$ and all of the $F(\sigma)$ are disjoint. This uniquely defines a hypergraph that would have this $f$ function.

    The total number of vertices in $V$ is equal to $n := \sum_{\sigma \in Nrv_w} f(\sigma)$ since each $f(\sigma)$ counts a set of vertices that are all disjoint, and since we are not allowing isolated vertices (which would not be counted in any nonempty $f(\sigma)$).
    WLOG, let $V := \{v_1, \ldots, v_n\}$. 
    Consider the set $\{\sigma \in \Nrv_w : f(\sigma) > 0\}$ and arbitrarily order those simplices $\sigma_0, \ldots, \sigma_N$.
    For $0\leq j \leq N$ define $s_j := \sum_{k < j} f(\sigma_k)$ and $s_{N+1} := \sum_{k=0}^N f(\sigma_k) = n$.
    Then assign vertices $v_{s_j+1}, \ldots, v_{s_{j+1}}$ to $F(\sigma_j)$. This assigns $s_{j+1} - (s_{j}+1)+1 = s_{j+1}-s_j = f(\sigma_j)$ unique vertices to $F(\sigma_j)$ as required. For any $\sigma \in \Nrv_w$ with $f(\sigma)=0$ we define $F(\sigma)=\emptyset$. 
    Since we have created disjoint $F(\sigma)$ sets from the $f(\sigma)$ values up to relabeling of vertices the discussion before this proof implies that we have uniquely determined the structure of a hypergraph.
    
    In summary, since $\H_1$ and $\H_2$ gave rise to the same weighted nerve they must give rise to the same $f$ function which uniquely defines a hypergraph up to relabeling of vertices. Therefore $\H_1 \cong \H_2$ as desired.
\end{proof}


\begin{example}\label{ex:wtd_nerve_ex}
In this example we will walk through the construction of $f$ and $F$, as well as taking $F$ to the hypergraph structure via the procedure in the proof of Proposition \ref{prop:unique_wtd_nerve}, for the hypergraph in Figure \ref{fig:wtd_nerve_ex}(a) and its weighted nerve in \ref{fig:wtd_nerve_ex}(b). First we give $\sigma(v)$ for all $v \in V$ and then the $f(\sigma)$ values and $F(\sigma)$ sets.
\begin{align*}
    \sigma(a) &= \{A, D, F\}\\
    \sigma(b) &= \{B, D, F\}\\
    \sigma(c_1) &= \{C_1, C_2, F\}\\
    \sigma(c_2) &= \{C_1, C_2, F\}
\end{align*}
We will not show the $f(\sigma)$ values (resp. $F(\sigma)$ sets) for all $\sigma$ as most of them are 0 (resp. $\emptyset$), but we will show all the simplicies with nonzero $f$ values and a couple of the ones with 0 values to understand the procedure. We will use the sum over $v \in \cap_{e \in \sigma} e$, rather than all $v\in V$, to simplify the computation.
\vspace{1em}
\begin{center}
\begin{tabular}{ccccc}
$\sigma$ & $\cap_{e \in \sigma} e$ & Sum pieces & $f(\sigma)$ & $F(\sigma)$\\\hline\hline
$ADF$ & $a$    & $\sigma(a) \stackrel{?}{=} ADF$ Yes & 1 & $\{a\}$\\\hline
$DF$  & $a, b$ & $\sigma(a) \stackrel{?}{=} DF$ No & 0 & $\emptyset$\\
      &        & $\sigma(b) \stackrel{?}{=} DF$ No &   & \\\hline
$BDF$ & $b$    & $\sigma(b) \stackrel{?}{=} BDF$ Yes & 1 & $\{b\}$\\\hline
$C_1C_2F$ & $c_1, c_2$ & $\sigma(c_1) \stackrel{?}{=} C_1C_2F$ Yes & 2 & $\{c_1, c_2\}$\\
          &            & $\sigma(c_2) \stackrel{?}{=} C_1C_2F$ Yes &  &  \\\hline
$C_1F$ & $c_1, c_2$ & $\sigma(c_1) \stackrel{?}{=} C_1F$ No & 0 & $\emptyset$\\
          &         & $\sigma(c_2) \stackrel{?}{=} C_1F$ No &   & \\\hline
\end{tabular}
\end{center}
\vspace{1em}

We can verify that $w(\sigma) = \sum_{\tau \leq \sigma} f(\tau)$ (again, we will not show this for all simplices):
\begin{center}
\begin{tabular}{cccc}
$\sigma$ & $\tau \supseteq \sigma$ & $\sum f(\tau)$ & $w(\sigma)$\\\hline\hline
$ADF$    & $ADF$                   & 1              & 1 \\\hline
$DF$     & $ADF, BDF, DF$          & 1+1+0          & 2 \\\hline
$D$      & $ADF, BDF, AD, BD, FD, D$ & 1+1+0+0+0+0  & 2 \\\hline
$F$      & $ADF, BDF, AF, DF, BF, $  & 1+1+0+0+0+   & 4 \\
         & $C_1F, C_2F, C_1C_2F, F$  & 0+0+2+0      &   \\\hline
\end{tabular}
\end{center}
\vspace{1em}

Lastly, we will go through the procedure of creating $F(\sigma)$ sets from the $f(\sigma)$ values as in the proof of Proposition \ref{prop:unique_wtd_nerve}. The sum of all $f(\sigma)$ values is 4 so we let $V:=\{v_1, v_2, v_3, v_4\}$. The set $\{\sigma \in \Nrv(\H) : f(\sigma) > 0\}$ is arbitrarily ordered as $\{ADF, BDF, C_1C_2F\} =: \{\sigma_0, \sigma_1, \sigma_2\}$. We compute 
\[s_0 := 0, \quad s_1 := 1, \quad s_2 := 1+1=2, \quad s_3 := 1+1+2 = 4.\]
From here we assign $F(\sigma_j) = \{v_{s_j+1}, \ldots, v_{s_{j+1}}\}$:
\begin{align*}
F(\sigma_0=ADF) &= \{v_{0+1}, \ldots, v_{1}\} = \{v_1\}\\
F(\sigma_1=BDF) &= \{v_{1+1}, \ldots, v_{2}\} = \{v_2\}\\
F(\sigma_2=C_1C_2F) &= \{v_{2+1}, \ldots, v_{4}\} = \{v_3, v_4\}.
\end{align*}
Finally, this gives us the following hypergraph structure following the discussion between the proofs of Lemma \ref{lem:unique_f} and Proposition \ref{prop:unique_wtd_nerve}:
\begin{align*}
    \sigma(v_1) &= \{A, D, F\}\\
    \sigma(v_2) &= \{B, D, F\}\\
    \sigma(v_3) = \sigma(v_4) &= \{C_1, C_2, F\}
\end{align*}
This is the same as the structure for the original hypergraph with $v_1 \mapsto a$, $v_2\mapsto b$, $v_3\mapsto c_1$, and $v_4\mapsto c_2$.
\end{example}

This property of being able to recover the hypergraph from its weighted nerve is in contrast to all of the other constructions we have explored in this survey. Indeed, there are many hypergraphs (those in the same hyperblock) that have the same upper closure. We see in the following examples that the same is true for the ResBS, RelBS, and polar complex if you consider only the structure of the resulting simplicial complex (not vertex or simplex labels), although it is not the hyperblocks that provide the equivalences for these constructions.

\begin{example}
    For the ResBS, observe that any hypergraph that only has toplexes has no edge containment relations and therefore its ResBS is a set of vertices, one for each toplex, with no edges or higher simplices. So, any two non-isomorphic hypergraphs that have the same number of edges, and all edges are toplexes, have the same ResBS.
    Other examples of groups of non-isomorphic hypergraphs with isomorphic ResBS can be constructed with more complex ResBS.
\end{example}

\begin{example}
    For the RelBS, hypergraphs with only toplexes also provide an example, but we will be more specific. Consider any two non-isomorphic hypergraphs each with $N$ toplexes of size 3 and no sub-edges. Independent of how the toplexes intersect their barycentric subdivisions with the missing subcomplexes identified to a point in each will be isomorphic to a wedge of 2-spheres, $\bigvee_{i=1}^N S_2$.
\end{example}

\begin{example}
The following two hypergraphs have same polar complex structure:
\[
    \H_1 = \big(\{a, b, c, d\}, \{ab, bc, bd\}\big), \qquad
    \H_2 = \big(\{\bar{a}, b, \bar{c}, \bar{d}), \{ \bar{a}b\bar{d}, b\bar{c}\bar{d}, \bar{a}b\bar{c}\}\big).\]
Namely, 
\[
\Gamma(\H_1) = \Gamma(\H_2) = \Delta\big(\{ ab\bar{c}\bar{d}, \bar{a}bc\bar{d}, \bar{a}b\bar{c}d\}\big).
\]
\end{example}

Likewise, it is also true for the chain complexes for embedded and path homology. Recall from Proposition~\ref{prop:path_constant_hyperblock} that the path homology of a hypergraph depends only on its maximal faces. 
The proof in \cite{GrigoryanJimenezMuranov2019} uses the observation that the path complex for a hypergraph is the same as the path complex for a sub-hypergraph consisting only of its maximal edges.
Therefore, just like the upper closure the path complex is constant across hyperblocks.
For embedded homology it's not so straightforward, as we see in the next example. Adding certain sub-edges, to stay in the same hyperblock, does not change the embedded homology chain complex. But there are other cases where the embedded homology chain complex varies within a hyperblock.

\begin{example}
    The embedded homology has a high bar for including hyperedges (or linear combinations of hyperedges) in $\infgroup_p$. They must be in the preimage under $\bdr_p$ of something in $\Dgroup_{p-1}$ and also be in $\Dgroup_p$. So finding an example where adding in some sub-edges doesn't change these $\Dgroup_p$ is not difficult. Consider $\H_1$ with edge set $\{abc, ab, a, b\}$ and $\H_2$ with edge set $\{abc, ab, bc, a, b\}$. In both cases $\infgroup_2=\langle 0 \rangle$, $\infgroup_1 = \langle ab \rangle$, and $\infgroup_0 = \langle a, b \rangle$. It's not until you add in the edge $ac$, which completes the $ab-ac+bc$ cycle, that you get $\infgroup_2 = \langle abc \rangle$, $\infgroup_1 = \langle ab, ab-ac+bc \rangle$, and $\infgroup_0=\langle a,b\rangle$.
\end{example}

It is also worth noting that a hypergraph cannot be reconstructed from its line graph, even one that is weighted by the number of vertices in each pairwise intersection. In fact, a hypergraph cannot be reconstructed from its weighted line graph together with the weighted line graph of its dual~\cite{everett2013dual,kirkland2018two}. However, Proposition~\ref{prop:unique_wtd_nerve} shows that when one adds in all of the multi-way intersections weighted by their size, there is enough information to recover the hypergraph.

Our weighted nerve is closely related to structures introduced in \cite{robinson2022cosheaf} and \cite{PhysRevE.106.034319}, though what we do with them is different. 
In \cite{robinson2022cosheaf}, Robinson introduces two weightings of the Dowker complex of a binary relation (analogous to the upper closure of the hypergraph rather than the nerve): the total weight and the differential weight. 
When a binary relation is interpreted as the incidence matrix of a hypergraph, the Dowker complex with total weight turns out to be the weighted nerve of the dual of the hypergraph. The differential weight is the same as the auxiliary weighting (the $f$ function) that we construct in the proof of Lemma~\ref{lem:unique_f}. Robinson proves that the relation is recoverable from both weighted Dowker complexes through algorithmic proofs, but he does not make the M\"{o}bius connection between them. 
Baccini et al.~\cite{PhysRevE.106.034319} also study two weightings of the upper closure, though the weights that play the role of Robinson's differential weights can be more general. This paper also comments that a reconstruction is possible. 
Robinson's work proceeds by defining a cosheaf and Baccini et al.~introduce a weighted Hodge Laplacian from the weighted simplicial complexes.
We leave it as an open question to explore the connection between the cosheaf, the weighted Hodge Laplacian, and the persistent homology filtration of the weighted nerve that we define next.

It follows directly from the definition of the weighted nerve that if $\sigma \subseteq \tau$,  $\sigma$ corresponds to fewer hyperedges and thus $w(\sigma) \geq w(\tau)$. This allows us to construct a filtration of the weighted nerve and define the weighted nerve persistent homology of a hypergraph.
\begin{definition}
    Let $\H$ be a hypergraph with no isolated vertices and $\Nrv_w(\H)$ be its weighted nerve. We create a filtration $X_n \subseteq X_{n-1} \subseteq \cdots \subseteq X_0$ where $X_k = \{\sigma \in \Nrv_w(\H) : w(\sigma) \geq a_k\}$ (for some threshold $a_k$ that is monotonically decreasing). The persistent homology of this filtration is the \emph{weighted nerve persistent homology} of hypergraph $\H$. 
    A persistence \emph{barcode} records the interval decomposition corresponding to the homology basis for each homological dimension $p$.
\end{definition}

\begin{example}[Figure \ref{fig:rest_bary_ex} running example]
To continue the example in Figure \ref{fig:wtd_nerve_ex}, we get the following filtration using threshold values $a_2=4, a_1=2, a_0=1$:
\[
 \{F\} \subseteq\{C_1, C_2, F, D, C_1C_2, C_1F, C_2F, C_1C_2F, DF\} \subseteq \Nrv_w(\H).
\]
Applying persistent homology to this filtration, we get a single bar $[4, -\infty)$ in the $\Hgroup_0$ barcode and no other higher dimensional structure.
\end{example}

In prior sections, we proved or cited conditions such that some of the simplicial and chain complex constructions are functorial.
In the case of the weighted nerve, we again reference Robinson's work where he proves that the construction of a specific cosheaf (and sheaf) from a relation, such that the (co)stalks determine the total weight function of the ``Dowker complex,'' is functorial from the category of relations to the category of (co)sheaves \cite[Theorems 4 and 5]{robinson2022cosheaf}.
While this seems to be promising for our weighted nerve (which is dual to the weighted upper closure), it is not clear that his result directly implies functoriality of the weighted nerve since the duality map is not a functor from $\MHyp$ to itself. 
So we cannot simply take the dual and then apply Robinson's functor.
However, since the hypergraph is fully recoverable from the weighted nerve there may be a proof, analogous to Robinson's, that goes directly to the weighted nerve.
Additionally, the fact that duality is a functor on $\Rel$, where Robinson's work lives, coupled with the fact that $\MHyp$ is a subcategory of $\Rel$ could be another direction.
We leave this as an open question.

\subsection{Homology theories for hypergraphs via categorification}
\label{sec:mag_chrom}

In recent decades, categorification has emerged as a powerful tool in mathematics, particularly in capturing subtle hidden structures overlooked by traditional enumerative methods. Formally, categorification involves elevating an 
$n$-category to an $(n+1)$-category, $n \geq 0$, akin to lifting the Euler characteristic of a topological space to its (singular) homology. Advantages of the singular homology include, but are not limited to, functoriality and distinguishing spaces with the same Euler characteristic. 
The opposite process is often referred to as decategorification. For example, Betti numbers can be thought of as a decategorification of homology. 

Categorification, first introduced and popularized in knot theory \cite{khovanov1999categorification, bar2002khovanov}, has inspired a number of categorifications in graph theory \cite{chandler2023strength, eastwood2007euler, HGR2005, HGRArb2005, sazdanovic2018categorification}. 
Categorifying a 1-variable polynomial with integer coefficients for graphs, such as the chromatic polynomial, amounts to constructing a homology theory such that the Betti number of the $i$-th homology group equals the $i$-th coefficient of the polynomial. 
While the Euler characteristic can be thought of as an integer-valued ``compression'' of a finitely-nonzero sequence of finite-dimensional
vector spaces over a field  by computing the alternating sum of dimensions \cite{ghrist2018persistent}, a categorification consists of constructing a meaningful sequence of vector spaces whose dimensions are equal to the integer coefficients of the polynomial. Two main examples of homology theories for hypergraphs that fit into this framework are the recently introduced chromatic and magnitude homology theories \cite{aslam2023categorifying,bi2024magnitude}, both motivated by active research on analogous invariants for graphs.
In the remainder of this section we provide a high level description and some properties for both of these categorification-based homology theories and point the reader to relevant papers for the details. 

\subsubsection{Chromatic homology}
Similar to the choice of group or field to serve as coefficients in the chain groups for standard homology, in chromatic homology for graphs \cite{eastwood2007euler, baranovsky2012graph} and simplicial complexes \cite{CDS2019} one must select either an algebra \cite{HGR2005, HGRArb2005} or a manifold to serve as the labels (or colors) for groups of hyperedges.
This provides a family of chromatic homology theories that have different properties depending on the choice of algebra or manifold.
The \emph{state sum chromatic homology} for hypergraphs was introduced in \cite{aslam2023categorifying} and represents one way to generalize chromatic homology from graphs and simplicial complexes to hypergraphs.
This is a bigraded homology where chain groups, $\Cgroup_{i,j}$, are generated by sub-hypergraphs with $i$ hyperedges and $j$ connected components, whose components are labeled (colored) by elements of the chosen algebra or manifold.
The boundary map from $\Cgroup_{i, j}$ to $\Cgroup_{i+1, j}$ is induced by tracking how labeled components are affected when adding individual hyperedges to the sub-hypergraph generators of $\Cgroup_{i,j}$ to obtain the sub-hypergraph generators of $\Cgroup_{i+1,j}$.
To fully define state sum chromatic homology for hypergraphs requires several additional concepts that can be found in \cite{aslam2023categorifying}, along with the hypergraph configuration spaces.

We briefly note some basic observations about chromatic homology in simple cases. If a hypergraph contains a hyperedge of size 1 then its chromatic homology is trivial. 
If $\H \in \MHyp$ has multi-edges then the chromatic homology of $\H$ is the same as the chromatic homology of $Col(\H)$, obtained by collapsing repeated edges into a single edge.
Additionally, chromatic homology theories yield strictly stronger invariants for hypergraphs than the chromatic polynomials they categorify, which is not always the case for graphs \cite{lowrance2017chromatic}. For example, one can construct a pair of hypergraphs with identical chromatic polynomials but distinct chromatic homology by adding a hyperedge containing all vertices to each hypergraph in a pair of cochromatic hypergraphs \cite{aslam2023categorifying}. 


\subsubsection{Magnitude homology}
The Euler characteristic for finite categories \cite{Leinster2008EulerChar} leads to a power series invariant of graphs called magnitude \cite{leinster_2019}.  Hepworth and Willerton categorified magnitude of a graph, referred to as the magnitude homology \cite{hepworth2017categorifying}, that inspired a geometric definition \cite{asao2021geometric}, Eulerian \cite{giusti2024eulerian}, blurred \cite{otter2022magnitude}, and persistent magnitude homology \cite{govc2021persistent}, whose structure and properties have been extensively analyzed \cite{asao2024girth, gu2018graph, kaneta2021magnitude, sazdanovic2021torsion, martin2025torsion}. The original construction by Hepworth and Willerton and the Eulerian magnitude homology by Giusti and Menara are generalized from graphs to digraphs by Huntsman \cite{huntsman2022magnitude, huntsman2023magnitude}, and to hypergraphs by Bi, Li, and Wu
\cite{bi2024magnitude}. 

The magnitude homology of hypergraphs can be introduced by using similar language to the path homology discussed in Section \ref{sec:path-homology}, and it is known that there is a spectral sequence from magnitude to path homology in the case of graphs \cite{asao2023magnitude}.  To define magnitude homology for hypergraphs Bi, Li, and Wu first generalized the notions of paths and distances between vertices of graphs to edges of hypergraphs. This enables them to compute path length with values in $\mathbb{Z}[\frac{1}{2}]$ since proper edge intersection of sequential edges counts as a full hop whereas edge containment counts as a half hop. 
The fact that sequences of \emph{hyperedges} generate magnitude chain groups for hypergraphs resembles the path homology of the dual hypergraph.  
The magnitude boundary map is roughly the same as the path homology boundary map, obtained by removing a hyperedge from the sequence in all possible ways, with the additional condition that only those paths obtained by removing a hyperedge  from the sequence that have the same length as the original path are in the image. In this same work the authors prove that the construction is functorial from a category of hypergraphs where morphisms preserve edge containment relations (this is neither $\Hyp$ nor $\MHyp$ but closer to a poset category) to the category of bigraded abelian groups.
The authors also define the concept of \emph{simple} magnitude for hypergraphs as an analogue of the regular path homology and a generalization of the magnitude for graphs. 


\subsection{Other topological notions of hypergraphs}
\label{sec:other}

In this paper, we have focused on simplicial, relative, and chain complex homology for hypergraphs. There are a variety of other methods in the literature to create topological structure from hypergraphs either from a topological or homological perspective. We briefly highlight a sampling of methods here for the interested reader to explore further. 

Deepthi and Ramkumart create open sets from hyperedges to form a topological space, but they do not consider homology \cite{deepthi2021hypergraph}. Diestel builds a homology theory for directed hypergraphs \cite{diestel2020homological}. In \cite{chung1992cohomological}, Chung and Graham consider cohomology rather than homology, with a particular focus on the space of all $\kappa$-uniform (i.e., every hyperedge has size $\kappa$) hypergraphs on a fixed vertex set rather than a single hypergraph. Emtander also restricts to $\kappa$-uniform hypergraphs as well as complete multi-partite hypergraphs, building simplicial complexes and computing Betti numbers from a Stanley-Reisner ring. The paper \cite{nguyen2020community} computes persistent homology on a hypergraph by instead constructing a metric from hypergraph data and \cite{liu2022neighborhood} builds a persistence theory parametrized by a weight function by combining embedded homology with a notion of a neighborhood hypergraph model, and then applies this to discern the structure of molecules.

\section{Building Intuition}
\label{sec:intuition}


Throughout Section \ref{sec:homology}, we provided a running example, computing all surveyed homology theories for the same hypergraph shown in Figure \ref{fig:rest_bary_ex}(a). 
This served to ground the theories in something real; now in this section we provide even more examples, limited to reporting betti numbers, in an attempt to better understand the nuances of each homology theory and build intuition. We do not claim to provide a complete characterization of any of these theories, but instead try to provide some insight into questions like: is the homology the same for all hypergraphs in a hyperblock? If the homology is not constant on the hyperblock, how does adding or removing sub-edges affect the homology? How does topological structure change by adding super-edges? Is there a relationship between the homology of a hypergraph and that of its dual?
Future work could tackle these questions more completely and prove which hypergraph modifications result in what kind of homology changes for each theory.

In Table \ref{tab:intuition}, we show the Betti numbers for several examples across multiple hyperblocks. 
In the following subsections, we make some remarks about the questions in the prior paragraph using examples from this table. 
Row 2 corresponds to our running example in Figure \ref{fig:rest_bary_ex}(a) and row $2^*$ is its dual.
We use publicly available Python code in HyperNetX\footnote{\url{https://github.com/pnnl/hypernetx}} (HNX) \cite{Praggastis2024} to compute simplicial homology of the upper closure ($\betti{\Delta}{\bullet}$) and polar complex ($\betti{pol}{\bullet}$). 
For ResBS and RelBS homology ($\betti{res}{\bullet}$ and $\betti{rel}{\bullet}$), we use an HNX module shared directly with us by Christopher Potvin that implements computational efficiencies proved in his thesis \cite{potvin2023hypergraphs}; this will be included in a future HNX release. 
We use preliminary code in a publicly available Julia package\footnote{\url{https://github.com/kylekoyanagi/PathHomology/}} written by Kyle Koyanagi to compute regular hypergraph path homology ($\betti{reg}{\bullet}$). 
We compute embedded homology ($\betti{emb}{\bullet}$) and weighted nerve persistent homology (shown in Table~\ref{tab:intuitionBarCode}) by hand. 
For examples and intuition on chromatic and magnitude homology see their respective papers.

\begin{table}
\resizebox{\columnwidth}{!}{
\begin{tabular}{c|c||cccc|cc}
Row \# & Hypergraph  & $\betti{\Delta}{\bullet}$ & $\betti{res}{\bullet}$ & $\betti{rel}{\bullet}$ & $\betti{pol}{\bullet}$ & $\betti{emb}{\bullet}$ & $\betti{reg}{\bullet}$\\
\hline\hline
1 & $abc, ab, bc$        & [1,0,0] & [1,0] & [0,1,0] & [1,0,0]  & [0,0,0] & [1,0] \\
2 & $abc, ab, a, b, c$   & [1,0,0] & [1,0,0] & [0,1,0] & [1,2,0] & [2,0,0] & [1,0] \\
\hline
$2^*$ & $ADF, BDF, CF$ & [1,0,0] & [3,0,0] & [0,1,2] & [1,0,0,0,0] & [0,0,0] & [1,0]\\
\hline
3 & $ab, bc$             & [1,0] & [2,0] & [0,2] & [1,0,0] & [0,0] & [1,0]\\
4 & $+a$          & [1,0] & [2,0] & [0,1] & [1,0,0] & [1,0] & [1,0]\\
5 & $+b$          & [1,0] & [1,0] & [0,1] & [1,0,0] & [1,0] & [1,0]\\
6 & $+a, c$       & [1,0] & [2,0] & [0,0] & [1,1,0] & [1,0] & [1,0]\\
7 & $+a, b$       & [1,0] & [1,0] & [0,0] & [1,0,0] & [1,0] & [1,0]\\
8 & $+a, b, c$    & [1,0] & [1,0] & [1,0] & [1,1,0] & [1,0] & [1,0]\\
\hline
9 & $abc, abd, acd, bcd$ & [1,0,1] & [4] & [0,0,4] & [1,0,1,0] & [0,0,0] & [1,0]\\
10 & $+ a$                & [1,0,1] & [2,0] & [0,0,2] & [1,0,1,0] & [1,0,0] & [1,0]\\
11 & $+ a, b$             & [1,0,1] & [1,1] & [0,1,1] & [1,0,2,0] & [2,0,0] & [1,0]\\
12 & $+ a, b, c$          & [1,0,1] & [1,3] & [0,3,1] & [1,0,4,0] & [3,0,0] & [1,0]\\
13 & $+ a, b, c, d$       & [1,0,1] & [1,5] & [0,5,1] & [1,0,7,0] &[4,0,0] & [1,0]\\
14 & $+ ab$               & [1,0,1] & [3,0] & [0,0,3] & [1,0,1,0] & [0,0,0] & [1,0]\\
15 & $+ ab, a, b$         & [1,0,1] & [1,0,0] & [0,0,1] & [1,0,1,0] & [1,0,0] & [1,0]\\
16 & $+ ab, bc, a, b, c$  & [1,0,1] & [1,1,0] & [0,1,1] & [1,0,2,0] & [1,0,0] & [1,0]\\
17 & $+ ab, bc, ac$       & [1,0,1] & [1,0]   & [0,0,1] & [1,0,1,0] & [0,0,0] & [1,0]\\
18 & $+ ab, bc, ac, a, b, c$ & [1,0,1] & [1,0,0] & [0,0,1] & [1,0,1,0] & [0,0,0] & [1,0]\\
\hline
19 & $abcde, abc, abd, acd, bcd$  & [1,0,0,0,0] & [1,0] & [0,0,0,5,0] & [1,0,0,1,0] & [0,0,1,0,0] & [1,0]\\
20 & $+ ab,bc,ac$     & [1,0,0,0,0] & [1,0,0] & [0,0,0,0,0] & [1,0,0,1,0] & [0,0,1,0,0] & [1,0]\\
21 & $+ abcd$          & [1,0,0,0,0] & [1,0,0] & [0,0,0,0,0] & [1,0,0,0,0] & [0,0,0,0,0] & [1,0]\\
\hline
22 & $abcde,bcdfg,abc,acd,abd$ & [1,0,0,0,0] & [2,0] & [0,0,0,3,1] & [1,0,0,1,0,0,0] & [0,0,0,0,0] & [1,0]\\
\end{tabular}
}
\caption{Table of Betti numbers for each homology theory across a variety of hypergraphs using $\K=\Z/2\Z$ coefficients. The horizontal lines break the table into distinct hyperblocks. The $+$ notation means we add the list of hyperedges in that row to the hypergraph in the first row of each hyperblock. For instance, in Row 14, we add hyperedge $\{a, b\}$ to the four 3-edges (2-faces) of the tetrahedron in Row 9. Row 2* corresponds to the dual of the hypergraph in Row 2. In each cell, the vector is $[\beta_0, \beta_1, \cdots, \beta_k]$ for some value $k \leq dim(K)$, where $K$ is the topological object we are taking homology of.}
\label{tab:intuition}
\end{table}

\paragraph{A note on computational complexity}
Although this survey is focused on the theoretical definitions and not the computational aspects of these homology theories, whether or not they are computable for real-world (large) hypergraphs may be a factor in which homology theory is chosen for a particular application. 
Therefore, we spend a little time here discussing complexity of these various homology theories.

The computation of simplicial homology, or chain complex homology with finite dimensional $\Cgroup_p$, amounts to reducing boundary matrices to their Smith normal form. Gaussian elimination can be performed if homology is computed over a field.
These algorithms are known to be polynomial in the number of rows and columns of the matrix, which is equivalent to the number of simplices in the simplicial complex or the sum of the dimensions of $\Cgroup_p$ (across all dimensions in both cases).
For example, persistent homology over a field is known to be $O(n^3)$ where $n$ is the number of simplicies \cite{edelsbrunner2000topological,zomorodian2005computing}.
Therefore, much of what we report here will be bounds on the number of simplices or dimensions of $\Cgroup_p$.
We will not discuss the time or space complexity of constructing the topological objects.
In all cases assume we have hypergraph $\H = (V, E)$ with $|V|=n$ and $|E|=m$ and we are computing homology up to dimension $D$. To compute homology up to dimension $D$ we need the simplicial complex up to dimension $D+1$ and the chain complex up to $\Cgroup_{D+1}$.

For the hypergraph upper closure the number of simplices up to dimension $D+1$ is bounded above by the sum of the sizes of the powerset of the edges. When edges intersect of course the principle of inclusion-exclusion will reduce this, but for a coarse upper bound we get the following sum:
\[|\Delta(\H)| \leq \sum_{d=1}^{D+2} \sum_{e \in E} {|e| \choose d}.\]
If the max edge size is $M$ we can get a simpler (but larger) upper bound:
\[|\Delta(\H)| \leq m\sum_{d=1}^{D+2} {M \choose d}.\]

The number of \emph{vertices} in the ResBS is $m$, the number of edges in $\H$. Computing the exact number of simplices in $\bary^\ECP_{res}$ would require knowing the lengths of all of the chains in $\ECP(\H)$. 
In the worst case scenario the hyperedges in $\H$ form a total order such that $e_1 \subseteq e_2 \subseteq \cdots \subseteq e_m$ and in this case there are $2^m$ total simplices in $\bary^\ECP_{res}$, or $\sum_{d=1}^{D+2} {m \choose d}$ up to dimension $D+1$.
Therefore, 
\[|\bary^\ECP_{res}(\H)| \leq \sum_{d=1}^{D+2} {m \choose d}.\]

To compute the RelBS homology via the original definition requires Gaussian elimination for the full barycentric subdivision and for the missing subcomplex. 
However, using a result in Potvin's thesis \cite[Theorem 5.2.3]{potvin2023hypergraphs} only requires the homology of the missing subcomplex and the upper closure, not the full barycentric subdivision.
A complexity analysis has not yet been performed for this algorithm and we leave that as an open question.

An upper bound for the number of simplices in the polar complex is straightforward since each maximal simplex is size $n$ and there are $m$ of them:
\[|\Gamma(\H)| \leq m\sum_{d=1}^{D+2} {n \choose d}. \]

Of all of the simplicial homology algorithms, we can see that ResBS requires the smallest topological object in terms of an upper bound on the number of simplices whereas the worst is likely the RelBS homology (even with Potvin's shortcut). 
However, our treatment here is admittedly simplistic and a more rigorous complexity analysis providing algorithms for each method would be of interest.

Moving to the chain complex based methods, embedded homology has been well-studied, including computational algorithms. In \cite{liu2024computing} the authors propose an algorithm for computation of (persistent) embedded homology over field coefficients with time complexity $O(m^4)$.
Complexity of path homology for graphs has also been rigorously analyzed \cite{chowdhury_persistent_2018,dey2022efficient}, but hypergraph path homology is more recent and less explored. 
Koyanagi shared with us that his algorithm provided in the software mentioned above has complexity $O(d^2 n^{d+2} m)$ for $d$-dimensional homology \cite{koyanagi2025personal}.

\subsection{Some initial intuition}
\label{sec:init_intuit}
We begin with the example in Row 19 of Table~\ref{tab:intuition} presented to us by Peter Bubenik. 
He posited that this example illustrates a hypergraph that should ``intuitively" have some homology in dimension 2. The hypergraph consists of the four 3-edges (2-faces) of a tetrahedron, together with a 5-edge containing all 4 vertices of the tetrahedron plus 1 additional vertex. One could consider this as a solid sphere (representing the 5-edge) with a tetrahedral void inside, which should have homology in dimension 2. Of course, the upper closure of this hypergraph is just the 4-dimensional simplex which has only $\betti{\Delta}{0}=1$ and no higher-dimensional homology. The ResBS of this example is also trivial as there is only one toplex. It is interesting that $\betti{rel}{3}=5$, so somehow quotienting by all of the missing sub-edges (of which there are many!) results in five 3-dimensional voids. But this is difficult to visualize. 
We also see in the table that $\betti{pol}{3}=1$\footnote{The generator is $\{abce, abde, acde, bcde, abc\bar{e}, abd\bar{e}, acd\bar{e}, bcd\bar{e}\}$ which is all 2-faces in $abcd$ union $e$ and all 2-faces in $abcd$ union $\bar{e}$}. 
Embedded homology is the only homology theory computed here that aligns with the intuition that $\betti{emb}{2}=1$.
We will come back to this example when we talk about adding sub-edges and super-edges. 

A suggestion from Bubenik was to create a filtration of simplicial complexes, $\{K_i(\H)\}$, where $K_i(\H) = \Delta(\{e \in \H : |e| \leq i\})$. In this example, $K_0 = \cdots = K_2 = \emptyset$, $K_3 = K_4 = \Delta(\{abc, abd, acd, bcd\})$, and $K_5 = \Delta(\H)$. The persistent homology of this filtration has a bar $[3,5)$ in dimension 2 so this filtration ``sees'' the dimension 2 homology. However, if edge $bcd$ is replaced by edge $bcdfg$ (adding new vertices $f$ and $g$), the tetrahedral void is both formed and killed at $K_5$ and so there is no bar in $\hom{}{2}$. This new edge replaces face $bcd$ with a larger edge and so whether this ``should'' have homology in dimension 2 is not clear. We show the Betti numbers for all other homology theories on Row 22 of the table and see that the feature that was in embedded homology in dimension 2 is no longer present. $\betti{res}{\bullet}$ also changes from row 19 to row 22 while $\betti{pol}{\bullet}$, $\betti{\Delta}{\bullet}$, and $\betti{reg}{\bullet}$ remain the same.

\subsection{Hyperblock variation}
In Section \ref{sec:closure-homology}, we observed that the closure homology is the same for all hypergraphs in a hyperblock since, by definition, all of these hypergraphs have the same upper closure.
Similarly, in Section \ref{sec:path-homology}, we noted that in \cite[Proposition 5.6]{GrigoryanJimenezMuranov2019} the authors prove hypergraph path homology depends only on the maximal hyperedges and is therefore also constant on hyperblocks.
However, it is clear from the table that $\hom{res}{\bullet}$, $\hom{rel}{\bullet}$, $\hom{emb}{\bullet}$, and $\hom{pol}{\bullet}$ are not constant on hyperblocks. In the next two subsections, we will make some observations about how adding sub- and super-edges can change the structure for these homology theories.

\subsection{Adding sub-edges}
When adding an edge, $e$, to a hypergraph, if there is already an edge $f$ such that $e \subseteq f$, we say that we are adding a \emph{sub-edge}. This does not change the hyperblock but may change the homology for some of these homology theories. Here we describe some intuition gleaned from Table \ref{tab:intuition} about how adding sub-edges affects the homology of the various topological objects.
\\
\paragraph{\textbf{Restricted barycentric homology}}
Recall from Proposition \ref{prop:res_max} that $\betti{res}{0}$ is maximized within a hyperblock by the simple hypergraph, the one with no sub-edges. If there are no sub-edges, then the edge containment poset is just the set of isolated vertices, one for each hyperedge, and so in addition to $\betti{res}{0}$ being maximized, $\betti{res}{k}=0$ for $k \geq 1$. As sub-edges are added, components can merge only if the added edge is a sub-edge of more than one edge. Structure can also be created in dimension 1 by adding sub-edges in a specific way. The canonical example is illustrated in moving from Row 9 to Row 11 in Table \ref{tab:intuition}. Adding edges $a$ and $b$ to the hypergraph that already contains edges $abc$ and $abd$ creates a cycle between those four hyperedges in $\ECP(\H)$, $a < abc > b < abd > a$, which translates to the 1-cycle $[a, abc]+[abc, b]+[b, abd]+[abd,a]$. In the general case, whenever there is an intersection between two hyperedges, $e \cap f \neq \emptyset$, you can create a 1-dimensional cycle in $\bary_{res}(\H)$ by adding two non-comparable sub-edges, $g$ and $h$, within $e \cap f$. Continuing this operation of adding two non-comparable edges into the intersection $g \cap h$ will lift this 1-cycle into a 2-dimensional cycle. Adding these two edges effectively creates a suspension of the 1-cycle, creating an octahedron. 
This process can be continued to form higher dimensional homological features.
Alternatively, we can destroy cycles by adding a sub-edge into the common intersection of all edges involved in the cycle, if an intersection exists. 
We thank Clara Buck for her insight into how adding sub-edges affects the ResBS homology.
\\
\paragraph{\textbf{Relative barycentric homology}}
In Potvin's thesis, he proved that if a hypergraph is simple then $\betti{rel}{k-1}(\H)$ is equal to the number of hyperedges of size $k$ \cite[Theorem 4.3.9]{potvin2023hypergraphs}. However, it's not clear what adding sub-edges can do to the structure in general. We see in Rows 9-18 that adding subedges to the tetrahedron can change the homology significantly in $\betti{rel}{\bullet}$. This is true also for $\betti{res}{\bullet}$ but we understand more about how that structure is formed and destroyed, as discussed above. It remains an open question to characterize what hypergraph structure cycles in $\hom{rel}{\bullet}$ correspond to, even in dimension 1.
\\
\paragraph{\textbf{Polar complex homology}} 
Little is known about the polar complex homology, but we observe in Table \ref{tab:intuition} that there is less variation in $\betti{pol}{\bullet}$ than in most of the other homology theories (except for $\betti{reg}{\bullet}$). We know that the total number of vertices in $\pol(\H)$ is $2|V|$, each maximal simplex is of size $|V|$, and there are $|E|$ such maximal simplices. If a sub-edge, $f \subset e$, is added to the hypergraph then $\Sigma(e) \cap \Sigma(f) = f \cup \overline{(V\setminus e)}$. But how this changes the polar complex homology likely depends heavily on the structure of $\pol(\H)$ and where the edge $f$ is added.
\\
\paragraph{\textbf{Embedded homology}}
In embedded homology, adding sub-edges helps to create structure, but it requires very specific sets of sub-edges. Given an edge $e$, if all of its sub-edges of size $|e|-1$ are present, then $e$ is present in $\inf_{|e|-1}$. But this is not the only way that $e$ can appear in $\inf_{|e|-1}$. For a set of edges $\{e_1, \ldots, e_k\}$ with $|e_1|=\cdots=|e_k|$ if all edges in $\bdr(e_1+\cdots+e_k)$ are in the hypergraph, then $e_1+\cdots+e_k$ is in $\inf_{|e_i|-1}$. Of course, making $\inf_k$ larger does not necessarily create homological structure in dimension $k$, but having $\inf_k$ be nonzero is required for structure to be present. Therefore, understanding how sub-edges affect construction of the chain complex is a good first step. While embedded homology is well-studied and applied, there remain open questions that will help build intuition and characterize the relationship between structure in $\H$ and structure in $\hom{emb}{\bullet}(\H)$.
\\
\paragraph{\textbf{Weighted nerve persistent homology}}
As the weighted nerve persistent homology corresponds to an interval decomposition of features rather than a single homology computation, we display the barcode results in Table~\ref{tab:intuitionBarCode}. From small examples (e.g., hypergraphs in Rows 1-2, 3-8, or 9-11 and 14), it might seem that adding sub-edges does not change the weighted nerve persistent homology. 
However, we note that the barcode for the hypergraph in Row 19, 
$\PH_0: \{[5,-\infty)\}, \PH_2: \{[1,-\infty)\}$, 
is different from the weighted nerve persistent homology of the hypergraph $\{abcde\}$, which is in the same hyperblock. The barcode for $\{abcde\}$ is just $[5,-\infty)$ in dimension 0 as the weighted nerve is just a 0-simplex with weight 5.
In fact, if you add sub-edges $\{abc, abd, acd\}$, the barcode of the weighted nerve persistent homology is still just $[5,-\infty)$ for $\PH_0$. But closing the inner tetrahedron creates some additional structure, which aligns with the intuition discussed in Section \ref{sec:init_intuit}.
In row 22 of Table \ref{tab:intuitionBarCode}, we see that the weighted nerve still captures nontrivial structure in dimension 2 when the second large edge is added. The weighted nerve is the only homology theory that captures dimension 2 structure in the row 22 example (while polar and RelBS homologies capture structure in dimension 3).
We have not yet delved into studying the types of sub-edges that do or do not change the weighted nerve persistent homology but doing so would be an interesting open question.

\begin{table}
\resizebox{\columnwidth}{!}{
\begin{tabular}{c|c||ccc}
Row \# & Hypergraph  & $\PH_0$ & $\PH_1$ & $\PH_2$ \\
\hline\hline
1 & $abc, ab, bc$        &  $\{[3,-\infty)\}$ & $\emptyset$ & $\emptyset$\\
2 & $abc, ab, a, b, c$   & $\{[3,-\infty)\}$ & $\emptyset$ & $\emptyset$ \\
\hline
$2^*$ & $ADF, BDF, CF$ & $\{[3,-\infty),[3,2)\}$ & $\emptyset$ & $\emptyset$\\
\hline
3 & $ab, bc$             & $\{[2,-\infty), [2,1)\}$ & $\emptyset$ & $\emptyset$ \\
4 & $+a$          & $\{[2,-\infty), [2,1)\}$ & $\emptyset$ & $\emptyset$ \\
5 & $+b$          & $\{[2,-\infty), [2,1)\}$ & $\emptyset$ & $\emptyset$ \\
6 & $+a, c$       & $\{[2,-\infty), [2,1)\}$ & $\emptyset$ & $\emptyset$ \\
7 & $+a, b$       & $\{[2,-\infty), [2,1)\}$ & $\emptyset$ & $\emptyset$ \\
8 & $+a, b, c$    & $\{[2,-\infty), [2,1)\}$ & $\emptyset$ & $\emptyset$ \\
\hline
9 & $abc, abd, acd, bcd$ & $\{[3,-\infty), [3,2) \times 3)\}$ & $\{[2,1) \times 3\}$ & $\{[1,-\infty)\}$\\
10 & $+ a$                &$\{[3,-\infty), [3,2) \times 3)\}$ & $\{[2,1) \times 3\}$ & $\{[1,-\infty)\}$\\
11 & $+ a, b$             & $\{[3,-\infty), [3,2) \times 3)\}$ & $\{[2,1) \times 3\}$ & $\{[1,-\infty)\}$\\
14 & $+ ab$               & $\{[3,-\infty), [3,2) \times 3)\}$ & $\{[2,1) \times 3\}$ & $\{[1,-\infty)\}$\\
\hline
19 & $abcde,abc, abd, acd, bcd$  & $\{[5,-\infty)\}$ & $\emptyset$ & $\{[1,-\infty)\}$\\
\hline
22 & $abcde,bcdfg,abc,acd,abd$ & $\{[5,-\infty),[5,3)\}$ & $\emptyset$ & $\{[1,-\infty)\}$\\
\end{tabular}
}
\caption{Table of barcodes for weighted nerve persistence across a variety of hypergraphs sampled from examples in Table~\ref{tab:intuition} (row numbering agrees).} 
\label{tab:intuitionBarCode}
\end{table}

\subsection{Adding super-edges}
While adding sub-edges can never change the hyperblock of the hypergraph, there are two types of super-edges one could consider. 
In one case, you could add a super-edge that contains one or more toplexes; this would change the hyperblock of the hypergraph.
Alternatively, you could add an edge $g$ that is a super-edge of a non-toplex and a sub-edge of a toplex; but this is also adding a sub-edge and would keep the hypergraph in the same hyperblock. We focus here on the former case rather than the latter.

We note that adding super-edges that contain one or more toplexes is somehow the reverse of adding sub-edges; therefore, much of what was said in the prior section has implications in this section. For example, in the case of ResBS, if your hypergraph consists of two non-comparable hyperedges, $e$ and $f$, and you add two new hyperedges, $g$ and $h$, together with two new vertices, $x$ and $y$, such that $g=e\cup f \cup \{x\}$ and $h=e \cup f \cup \{y\}$, then you will create a 1-dimensional cycle in $\bary_{res}(\H)$. Of course, this is not the only way to add super-edges; in fact, it is quite specific and requires the addition of two super-edges (and two vertices). In general, adding a single edge has the effect of ``coning off'' any topological structure contained in the ResBS of the set of edges it contains. 

The example in the prior section for weighted nerve persistent homology is interesting in the case of adding super-edges as well. The barcode for the hollow tetrahedron hypergraph $\{abc, abd, acd, bcd\}$ (without singleton and pairwise edges, Row 9) is $\PH_0=\{[3,2)\times 3, [3,-\infty)\}$, $\PH_1=\{[2,1] \times 3\}$, and $\PH_2=\{[1,-\infty)\}$. When $\{abcde\}$ is added (Row 19), the simplex corresponding to this large edge connects to all of the sub-edges which pulls together the components that start at $k=3$. This super-edge also kills the dimension 1 features, but as noted above it does not kill the dimension 2 feature. 

For the RelBS, Potvin again proved some helpful theorems. He introduces the notion of a \emph{maximum edge hypergraph} which is one that contains an edge $e=V$. There may be many (or few) other hyperedges, but it at least has one containing all vertices. In this case, he proves that $\hom{rel}{n}(\H) \cong \hom{\Delta}{n-1}(\missing(\H))$ for $n \geq 2$, where we recall that $\missing(\H)$ is the missing subcomplex \cite[Theorem 4.2.3]{potvin2023hypergraphs}. This does not tell us how the structure changes with and without the maximum edge, but it does tell us how to compute the RelBS homology when a maximum edge is added.

Even less is known about how adding super-edges affects the path, polar complex, and embedded homologies. We encourage readers to consider studying these theories to help develop more understanding.

\subsection{Duality}
Duality is an operation that does not modify the hypergraph, but exchanges the roles of vertices and hyperedges. As we observed in Section \ref{sec:closure-homology}, the closure homology of a hypergraph is isomorphic to the closure homology of its dual by Dowker duality. One might ask the question: is it true in general that the $X$-homology of a hypergraph is isomorphic to the $X$-homology of its dual (where $X$ is any of the homology theories introduced in this survey, or others developed in the future)? In fact, this is not generally true, and our running example in Figure \ref{fig:rest_bary_ex} is a counterexample for $\hom{res}{}, \hom{rel}{}, \hom{pol}{}$, $\hom{emb}{}$, and for the persistent homology of the weighted nerve.

So, while it is generally not true that the $X$-homology of a hypergraph is isomorphic to the $X$-homology of its dual, one could instead ask the weaker question of whether or not there is any relationship between these two homologies. Perhaps there are some conditions on the hypergraph such that you can achieve an isomorphism. Or there could be a way to bound the difference between the $X$-homology of a hypergraph and the $X$-homology of its dual (perhaps with certain assumptions on hypergraph structure). This provides an interesting open question. However, as noted earlier, duality is not a functor from $\Hyp$ or $\MHyp$ to itself so one might expect that there would be no such relationship. Indeed it could be that for every $n > 0$ you can find an $\H$ such that $|\betti{\bullet}{k}(\H) - \betti{\bullet}{k}(\H^*)| > n$.

For $\PH(\Nrv_w(\H))$, we provide a pathological example to show that the number of bars in $\PH(\Nrv_w(\H^*))$ can be arbitrarily more than the number of bars in $\PH(\Nrv_w(\H))$. Consider our running example from Figure \ref{fig:rest_bary_ex}(a), but instead of the singleton edge with vertex $c$, we inflate that edge to contain $c_1, \ldots, c_n$, as shown in Figure \ref{fig:pathological_wtd_nrv} along with its dual and their weighted nerves. Since the two nerves are homotopy equivalent (without weighting), the number of bars that continue to $-\infty$ must be the same. However, the barcode for $\Nrv_w(\H)$ in dimension 0 consists of one bar $[n+2, -\infty)$ while the barcode for $\Nrv_w(\H^*)$ in dimension 0 has bars $\{[3,-\infty), [3,2), [2,1)\times n\}$. Neither have persistent homology in dimension $k > 0$. By duplicating a vertex arbitrarily many times (equivalently, duplicating an edge), we can make the two barcodes arbitrarily different in number of bars. This raises an open question of whether requiring the hypergraph to be both vertex and edge \emph{collapsed} (i.e., in $\Hyp$ and with no duplicate vertices) allows us to bound or describe a relationship between the persistent homology of the weighted nerve of the hypergraph and that of its dual.


\begin{figure}[h!]
  \centering
\begin{subfigure}{.45\textwidth}
  \centering
  \resizebox{.9\textwidth}{!}{
  \begin{tikzpicture}[scale=0.8, ve/.style={fill=black,circle,scale=0.5},  he1/.style={draw = blue, line width=0.5mm, rounded corners,inner sep=3pt}]
    %
    \node[ve, color=ForestGreen,label={\textcolor{ForestGreen}{$c_1$}}] (vc1) at (0, 0) {};
  \node[ve, color=ForestGreen,label={\textcolor{ForestGreen}{$c_2$}}] (vc2) at (1, 0) {};

 \node[ve, color=ForestGreen,label={\textcolor{ForestGreen}{$c_n$}}]  (vcn) at (5, 0) {};
\node[ve, color=Brown, label={\textcolor{Brown}{$a$}}] (vA) at (0, -2) {};
\node[ve,color=Purple, label={\textcolor{Purple}{$b$}}] (vB) at (5, -2) {};    
\node [he1, color=orange, inner sep=10pt, fit=(vA)] {};
\node [he1, color=teal, inner sep=10pt, fit=(vB)] {};
 \node[color=orange] at (0.8,-2.2) {A};
 \node[color=teal] at (4.2,-2.2) {B};
 
\node [he1, inner sep=11,  fit=(vc1) (vcn)] {};
 \node[color=olive] at (2,0) {...};
 \node[color=blue] at (3,-0.3) {C};
 
\node [he1,draw = cyan,inner sep=17pt, fit=(vc1) (vcn) (vB) (vA)]{};
 \node[color=cyan] at (5.4, -0.9){D};
 \node [he1,draw = red, inner sep=14pt, fit=(vA) (vB)]{};
 \node[color=red] at (2, -1.5){{E}};
\end{tikzpicture}
    }
    \caption{$\H$}
  \label{fig:7TopLeft}
\end{subfigure}
  \begin{subfigure}{.4\textwidth}
  \centering
  \resizebox{0.9\textwidth}{!}{
  \begin{tikzpicture}[scale=0.6, ve/.style={fill=black,circle,scale=0.35},  he1/.style={draw = blue, line width=0.4mm, rounded corners,inner sep=2pt}]

 \node[ve, color=blue] (vC) at (0.8, 0) {};
 \node[color=blue] at (1, -0.35){\small{C}};
 \node[ve, color=orange] (vA) at (-1, -1.5) {};
  \node[color=orange] at (-0.7, -1.5){\small{A}};
 \node[ve, color=cyan] (vD) at (0.5, -2) {};
  \node[color=cyan] at (0.8, -1.9){\small{D}};
  \node[ve, color=red,label={}] (vE) at (2.2, -2) {};
   \node[color=red] at (1.8, -1.9){\small{E}};
\node[ve, color=teal] (vB) at (3.5, -1) {};
 \node[color=teal] at (3.5, -1.4){\small{B}};
 \node [he1, color=Brown, fit=(vA)(vE)] {};
    \node[color=Brown] at (-1, -1){\small{a}};
  \node [he1, draw=Purple,inner sep=5pt, fit=(vD)(vB)] {};
   \node[color=Purple] at (4.2,-1.2) {b};
    \node [he1, draw=ForestGreen,line width=0.6mm,inner sep=8pt, fit=(vD)(vC)] {};
\node[color=blue] at (2.8,0) {\tiny{\textcolor{ForestGreen}{$c_1,c_2,\ldots, c_n$}}};
\end{tikzpicture}
    }
    \caption{$\H^*$}
  \label{fig:7TopRight}
\end{subfigure}
\\
 \begin{subfigure}{.45\textwidth}
  \centering  \resizebox{0.7\textwidth}{!}{
  \begin{tikzpicture}[ve/.style={fill=black,circle,scale=0.5},   he1/.style={draw = blue, line width=0.5mm, rounded corners=5pt,inner sep=3pt}
];
      \draw[line width=0.3mm, color=black,fill=gray!15] (0, -1)--(1, -1.5)--(1, -0.5)--cycle;
      \node[color=black]  at (0.5,-0.6) {\tiny{1}};
      \node[color=black]  at (0.5,-1.4) {\tiny{1}};
      \draw[line width=0.3mm,color=black,fill=gray!15] (2, -1)--(1, -1.5)--(1, -0.5)--cycle;
        \node[color=black]  at (1.5,-0.6) {\tiny{1}};
       \node[color=black]  at (1.5,-1.4) {\tiny{1}};
    \node[ve, color=blue] (v3) at (1, 0) {};
    \node[ve, color=orange] (v1) at (0, -1) {};
   \node[ve, color=teal] (v2) at (2, -1) {};
   \node[color=orange]  at (-0.25,-1){\tiny{1A}};
  \node[color=teal]  at ([shift={(1:0.15)}]v2.10){\tiny{1B}};
   \node[color=blue]  at ([shift={(10:0.25)}]v3){\tiny{nC}};
    \node[ve, color=red] (v12) at (1, -1.5) {};
     \node[color=red, below] at ([shift={(40:0.19)}]v12.10){\tiny{2E}};
 \draw[line width=0.3mm, color=black] (v3)--(v12); 
\node[ve, color=cyan, label={}] (v123) at (1, -0.5) {}; 
\node[color=cyan]  at ([shift={(2:0.45)}]v123.10){\tiny{(n+2)D}};
\node[color=black]  at ([shift={(1:0.5)}]v1.10){\tiny{1}};
\node[color=black]  at ([shift={(1:1.3)}]v1.10){\tiny{1}};
\node[color=black]  at ([shift={(1:1)}]v1.10){\tiny{2}};
\node[color=black]  at (1.1,-0.2) {\tiny{n}};
\end{tikzpicture}}
 \caption{$\Nrv_w(\H)$}
  \label{fig:7DownLeft}
\end{subfigure}
 \begin{subfigure}{.45\textwidth}
  \centering  \resizebox{0.7\textwidth}{!}{
  \begin{tikzpicture}[ve/.style={fill=black,circle,scale=0.4},   he1/.style={draw = blue, line width=0.5mm, rounded corners=5pt,inner sep=3pt}
];
 \draw[line width=0.3mm, color=black,fill=gray!25] (0, -1.5)--(1.3, -1)--(-0.5, -0.2)--cycle;
  \draw[line width=0.3mm, color=black,fill=gray!15] (0, -1.5)--(1.5, -1.5)--(0.3, -0.2)--cycle;
  \draw[line width=0.3mm, color=black,fill=gray!15] (0, -1.5)--(1.5, -1.5)--(1.8, -0.2)--cycle;
  
\node[ve, color=Brown]  at (0,-1.5){};
\node[ve, color=Purple]  at (1.5, -1.5){};
 \node[color=Brown, inner sep=5pt]  at (-0.25, -1.5){\tiny{3a}};
 \node[color=Purple, inner sep=5pt]  at (1.7, -1.5){\tiny{3b}};
 \node[color=black]  at (0.8,-1.65) {\tiny{2}};
 
\node[ve, color=ForestGreen] (v1) at (-0.5, -0.2) {};
\node[black]  at (-0.5, 0){\tiny{\textcolor{ForestGreen}{$2c_1$}}};
\node[ve, color=ForestGreen] (v2) at (0.3, -0.2) {};
\node[black]  at (0.3, 0){\tiny{\textcolor{ForestGreen}{$2c_2$}}};
\node[ve, color=ForestGreen] (v3) at (1.8, -0.2) {};
\node[black]  at (1.8, 0){\tiny{\textcolor{ForestGreen}{$2c_n$}}};

\node[color=black]  at (-.15, -0.55){\tiny{1}};
\node[color=black]  at (0.4, -0.55){\tiny{1}};
\node[color=black]  at (1.55, -0.55){\tiny{1}};  

\node[color=black]  at (-0.4,-0.8) {\tiny{1}};
\node[color=black]  at (0.05,-0.8) {\tiny{1}};
\node[color=black]  at (1.8,-0.8) {\tiny{1}};

\node[color=black]  at (-0.1,-0.3) {\tiny{1}};
\node[color=black]  at (0.6,-0.3) {\tiny{1}};
\node[color=black]  at (1.5,-0.3) {\tiny{1}};
\end{tikzpicture}}
\caption{$\Nrv_w(\H^*)$}
  \label{fig:7DownRight}
\end{subfigure}

 \caption{Pathological example for duality and the persistent homology of the weighted nerve.}
    \label{fig:pathological_wtd_nrv}
\end{figure}
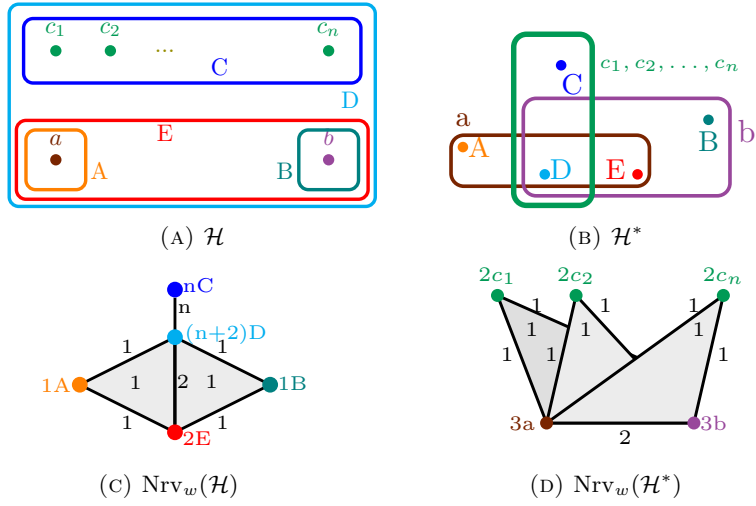

\section{Discussion}
\label{sec:discussion}

This survey provides descriptions and known properties of nine homology theories for hypergraphs introduced and studied in recent years. We also included an intuition building section containing observations that come  from computations on relatively small examples. 
The goal of this survey was to show that the research on homology theories for hypergraphs is active and to collect recent works in one place to observe similarities and differences among them.
In doing so, we have identified  some interesting open problems along the way; many are noted in the sections above.
Beyond those identified above we leave readers with a few closing thoughts of where work in homology theories for hypergraphs could proceed.

\subsection*{Persistence and stability}
In our survey, we include only one persistent homology theory for hypergraphs in Section \ref{sec:weighted}. The weights on the upper closure are induced by the structure of the hypergraph. However, because of the functoriality results, many of the other theories can also be considered in a persistence context through the introduction of any kind of function on hyperedges or vertices. Consider, for example, the ResBS homology. If we have a hypergraph $\H=(V,E)$ with a function on hyperedges $w : E \rightarrow \R$, we can induce a filtration $\H_\alpha = (V, E_\alpha)$ where $E_\alpha = \{e \in E : w(e) \geq \alpha\}$. It is clear that $E_\alpha \subseteq E_\beta$ whenever $\alpha \leq \beta$ and so $\H_\alpha \rightarrow \H_\beta$ is a morphism in $\Hyp$ (or $\MHyp$). By functoriality of $\hom{res}{\bullet}$, $\{\H_\alpha \hookrightarrow \H_\beta\}$ induces a persistence module. This opens up many new questions like stability to perturbations of the function and interpretation of persistent homology for all these theories. 

\subsection*{Characterizing structure through homology}
One reason that a practitioner may want to study homology of a hypergraph is to get an understanding of some aspect of hypergraph structure that traditional hypernetwork sciences measures cannot quantify, for example to identify complex substructures. In \cite{jenne2023stepping}, the authors propose that homology for hypergraphs, specifically for the ResBS, can generalize graph motifs to hypergraphs in some way. The straightforward generalization of motifs from graphs to hypergraphs results in a combinatorial explosion (there are already nearly 2,000 hypergraph motifs with only 4 hyperedges \cite{lee2020hypergraph}), but they argue that allowing small patterns to be unified if they have the same homological structure could be a way of combating that explosion. Understanding what kinds of small substructures are captured by each homology theory would be an interesting direction to pursue with potential application to hypergraph classification as in \cite{milo2002network} for graphs.

\subsection*{Practical applications}
Though it was noted that homology of hypergraphs could be useful for understanding structure of real world hypergraphs, there has been little work in that direction outside of the embedded homology work cited in Section \ref{sec:embedded} and \cite{aktas2023hypergraph}. A study in which each of these homology theories are computed and evaluated for hypergraphs coming from different application domains---e.g., biology, computer network data, collaborations---would be extremely valuable to the community. Achieving this requires developing additional computational tools. We believe that they should be released as open source, or incorporated into existing open source projects, in order for the community to take advantage. We point to projects like HyperNetX and Open Applied Topology\footnote{\url{https://openappliedtopology.github.io/}} as potential packages where these tools could be incorporated.


\section*{Acknowledgements}

The authors thank Robby Green, Clara Buck, Peter Bubenik, and Christopher Potvin for their helpful insight into concepts and proofs within this paper. In particular, we thank Robby Green for a helpful discussion about the proofs of Propositions~\ref{prop:closure_functor} and \ref{prop:bary_res_functor} and Peter Bubenik for the example in Row 19 of Table \ref{tab:intuition}.
We thank Kyle Koyanagi for reaching out and pointing us to his path homology code.
We also thank the anonymous reviewers for their insightful comments and questions that significantly improved this paper.

We are grateful for the Women in Computational Topology (WinCompTop) workshop for initiating our research collaboration. In particular, participant travel support was made possible through the grant National Science Foundation (NSF) DMS-1619908. This collaborative group was also supported by the American Institute of Mathematics (AIM) Structured Quartet Research Ensembles (SQuaRE) program, and the Institute for Mathematics and its Applications (IMA). This work was partially based upon work supported by the NSF under Grant No. DMS-1929284 while the authors were in residence at the Institute for Computational and Experimental Research in Mathematics in Providence, RI, during the Persistent Topology of Hypergraphs Collaborate@ICERM.

RS was partially supported by the NSF grant DMS-1854705.  BW was supported in part by grants from the NSF (DMS-2301361) and Department of Energy (DoE) (DE-SC0021015). YW was supported in part by grants from the NSF (CCF-2310411 and CCF-2112665). LZ was supported in part by grants from the NSF (CDS\&E-MSS-1854703 and BCS-2318171).

\bibliographystyle{plain}
\bibliography{hyperhom}

\end{document}